\documentclass{svproc}
\usepackage{url}
\usepackage{subcaption}

\usepackage{amsmath}
\usepackage[left=‘,right=']{dirtytalk}
\MakeRobust{\say}

\usepackage{tikz}
\usetikzlibrary{decorations.markings,arrows,decorations.text}

\DeclareSymbolFont{cmbrightop}{OT1}{cmbr}{m}{n}
\DeclareMathSymbol{\mPi}{\mathalpha}{cmbrightop}{5}

\DeclareMathOperator{\trace}{tr}

\newcommand{\Psih}{\Psi_h}

\renewcommand{\P}{P} 
\newcommand{\Q}{Q}

\newcommand{\dis}{\displaystyle}
\newcommand{\uh}{u_h}
\newcommand{\vh}{v_h}

\newcommand{\intQ}{\dis \int_\Q}
\newcommand{\intdQ}{\dis \int_{\partial\Q}}
\newcommand{\mintQ}{\displaystyle\dfrac{1}{|\Q|}\intQ}
\newcommand{\mintdQ}{\displaystyle\dfrac{1}{|\partial\Q|}\intdQ}
\newcommand{\intP}{\dis\int_P}

\newcommand{\dx}{\,\text{d}\boldsymbol{{x}}}  
\newcommand{\ds}{\,\text{d}s}

\newcommand{\myphi}{\varphi}

\newcommand{\phii}{\myphi_i}
\newcommand{\phij}{\myphi_j}
\newcommand{\phik}{\myphi_k}
\newcommand{\phil}{\myphi_\ell}

\newcommand{\n}{\boldsymbol{n}}

\renewcommand{\ni}{\n_i}

\newcommand{\ei}{\e{i}}
\newcommand{\eim}{\e{i-1}}

\newcommand{\dRi}{\dR{i}}
\newcommand{\dRj}{\dR{j}}
\newcommand{\di}{\d{i}}
\renewcommand{\dj}{\d{j}}

\newcommand{\NV}{{N_\P}}

\newcommand{\dof}{\text{dof}}

\newcommand{\PZ}{P_0}
\newcommand{\PZt}{\widetilde{P}_0}

\newcommand{\Vi}{V_i}  

\newcommand{\Vip}{V_{i+1}}

\newcommand{\vK}{\boldsymbol{\kappa}}
\newcommand{\vI}{\boldsymbol{I}}

\newcommand{\vgamma}{\boldsymbol{\gamma}}

\newcommand{\x}[1]{x_{#1}}
\newcommand{\y}[1]{y_{#1}}

\renewcommand{\phi}[1]{\myphi_{#1}}

\newcommand{\PN}[1]{\mPi^\nabla_{#1}}
\newcommand{\PI}{\text{I}}

\newcommand{\PZZ}{\mPi^0_0}

\renewcommand{\d}[1]{\boldsymbol{d}_{#1}}
\newcommand{\dR}[1]{\boldsymbol{d}^\perp_{#1}}

\newcommand{\e}[1]{\boldsymbol{e}_{#1}}

\newcommand{\vx}{\boldsymbol{x}}
\newcommand{\va}{\boldsymbol{a}}

\newcommand{\mK}{\mathsf{K}}
\newcommand{\mA}{\mathsf{A}}
\newcommand{\mB}{\mathsf{B}}
\newcommand{\mC}{\mathsf{C}}

\newcommand{\tauVEM}{\tau_{\textsc{\tiny VEM}}}

\usetikzlibrary{arrows}
\usetikzlibrary{calc}
\newcommand{\tikzAngleOfLine}{\tikz@AngleOfLine}
\def\tikz@AngleOfLine(#1)(#2)#3{%
  \pgfmathanglebetweenpoints{%
    \pgfpointanchor{#1}{center}}{%
    \pgfpointanchor{#2}{center}}
  \pgfmathsetmacro{#3}{\pgfmathresult}%
}

\newcommand{\mKVEM}{\mK_{\textsc{VEM}}}
\newcommand{\mKc}{\mK_{\textsc{C}}}
\newcommand{\mKs}{\mK_{\textsc{S}}}
\newcommand{\calS}{\mathcal{S}}

\begin{document}
\mainmatter              
\title{Quantitative study of the stabilization 
        parameter \\ in the virtual element method}
        

%
\titlerunning{Stability parameter in the VEM}  
%
\author{Alessandro Russo\inst{1,2} \and N. Sukumar\inst{3}}
\authorrunning{Alessandro Russo et al.} 
%
\tocauthor{Alessandro Russo, and N. Sukumar}
\institute{University of Milano-Bicocca, 20133 Milano, Italy,
\and
IMATI-CNR, 27100 Pavia, Italy\\
\email{alessandro.russo@unimib.it}\and
University of California, Davis CA 95616, USA\\
\email{nsukumar@ucdavis.edu}}

\maketitle              

\begin{abstract}
The choice of stabilization term is a critical component of the virtual element method (VEM).
However, the theory of VEM provides only asymptotic guidance for selecting the stabilization term, which ensures convergence as the mesh size approaches zero, but does not provide a unique prescription for its exact form.  
Thus, the selection of a suitable stabilization term is often 
guided by 
numerical experimentation and analysis of the resulting 
solution, including factors such as stability, accuracy, and efficiency.
In this paper, we establish a new 
link between VEM and 
generalized barycentric coordinates,
in particular 
isoparametric finite elements as a specific case. This connection enables the interpretation of the 
stability as the energy of a particular function in the 
discrete space, commonly known as the `hourglass mode.' Through this approach, this study 
sheds light on how the virtual element 
solution depends on the stabilization term, providing insights into the behavior of the method in more general 
scenarios.

\end{abstract}

\section{Introduction}

The virtual element method (VEM)~\cite{volley,projectors} is a stabilized Galerkin method that is accurate and robust on polygonal and polyhedral meshes.  The first-order VEM on simplices is identical to linear finite elements. Polygonal finite elements (see~\cite{Hormann:2017:GBC}) are based
on generalized barycentric coordinates such as
Wachspress basis (shape) functions~\cite{Wachspress:2016:RBG}
and mean value coordinates~\cite{Floater:2003:MVC,Floater:2015:GBC}. On
a quadrilateral, isoparametric finite element
shape functions are also an instance of GBCs. 

In this paper, we present new results over the quadrilateral that provide clearer connections of the finite element
method (FEM) and polygonal FEM to the virtual element method. A stabilization parameter is needed to ensure that the stiffness matrix in the VEM is consistent and stable (invertible). As noted in~\cite{hourglass}, this mirrors the development of hourglass finite elements over the four-node quadrilateral~\cite{Flanagan:1981:USH,TB:book}.
We first show that the stiffness matrix for the diffusion equation on any (convex or nonconvex) 
quadrilateral can be written as the sum of two 
contributions: a consistency matrix 
$\mA$ that is
exactly computable and a stabilization matrix that has
the form $\tau \mB$, where 
$\mB$ is known and the scalar $\tau$
is in fact the hourglass function associated with the 
shape functions of the four-node quadrilateral~\cite{hourglass} (see Section 2). In
Section 3, we compute values for $\tau$ on the square
and parallelogram for isoparametric FEM and Wachspress shape
functions. The decomposition of the element stiffness matrix
in the VEM is precisely of the form $\mA
+ \tau \mB$, where $\tau$ is set to 1, which is
elaborated in Section 4.  We present two numerical examples
in Section 5, and show that the standard
value of $\tau$ in the VEM
leads to a convergent scheme for
the diffusion equation.
  
\section{GBCs on a quadrilateral for the diffusion
         equation}\label{sec:Q4-Poisson}

Let $\Q$ be a quadrilateral with vertices 
$\Vi=(x_i,y_i)$, $i=1,2,3,4$, and let $\{\phii\}_{i=1}^4$
be a set of generalized barycentric coordinates such as
isoparametric bilinear FEM, harmonic, Wachspress or mean value coordinates.
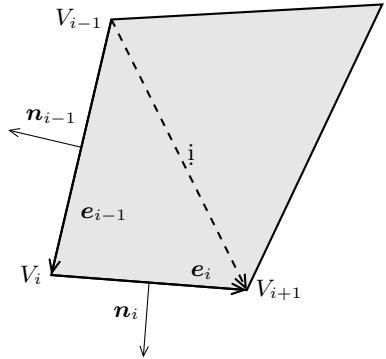
\begin{figure}[ht]
\begin{center}
\begin{tikzpicture}[scale=2]
\coordinate (v1) at (0.1,0.1);
\coordinate (v2) at (-0.5,-1);
\coordinate (v3) at (2,-0.7);
\coordinate (v4) at (2.2,0.3);
\coordinate (v5) at (1,2);
\coordinate (v6) at (-0.8,1.9);
\coordinate (v7) at (-1.2,0.2);
\coordinate (v67) at ($0.5*(v6)+0.5*(v7)$);
\coordinate (v71) at ($0.5*(v7)+0.5*(v1)$);
\coordinate (v16) at ($0.5*(v1)+0.5*(v6)$);
\draw [fill=gray!20] [thick] (v1) -- (v5) -- (v6) -- (v7) -- cycle;
\draw [thick] [-angle 45] (v6) -- (v7);
\draw [thick] [-angle 45] (v7) -- (v1);
\node [right] at (v1) {$V_{i+1}$};
\node [left] at (v6) {$V_{i-1}$};
\node [left] at (v7) {$V_i$};
\draw [thick] [dashed] [-angle 45] (v6) -- (v1);
\node [right] at (v16) {$\di$};
\tikzAngleOfLine(v6)(v7){\angle};
\draw[-angle 45] (v67) -- ++(\angle-90:0.5);
\tikzAngleOfLine(v7)(v1){\angle};
\draw[-angle 45] (v71) -- ++(\angle-90:0.5);
\node at ($(v1)+(-0.3,0.12)$) {$\ei$};
\node at ($(v7)+(0.35,0.4)$) {$\e{i-1}$};
\node at ($(v67)+(-0.2,0.2)$) {$\n_{i-1}$};
\node at ($(v71)-(0.15,0.2)$) {$\n_{i}$};
\end{tikzpicture}
\end{center}
\caption{The quadrilateral $\Q$.}
\label{fig:Q}
\end{figure}
Let $\vK$ be a constant symmetric positive-definite 
$2\times 2$ matrix
on the element $\Q$.
Then the element
stiffness matrix for the diffusion operator  
is defined by
\begin{equation}
\label{eq:stiffness-matrix}
\mK_{ij}:=\intQ\vK\nabla\phij\cdot\nabla\phii\dx.
\end{equation}
We will prove the following structure theorem:
\begin{theorem}
The matrix $\mK$ can be written as
\begin{equation}
\label{eq:identity}
\mK = \mA + \tau \mB ,
\end{equation}
where
\begin{itemize}
\item[$\bullet$]
$\mA$ is a $4\times 4$ matrix that depends only on the geometry 
of the quadrilateral $\Q$ and on the diffusion matrix $\vK$;
\item[$\bullet$]
$\mB$ is a $4\times 4$ matrix that depend only on the
geometry of the quadrilateral $\Q$; furthermore, 
$\mB$ is the same matrix for all parallelograms;
\item[$\bullet$]
$\tau$ is the energy of the function $\Psih$
of our local space
whose value at vertex $\Vi$ is 
$\displaystyle\dfrac{(-1)^i}{2}$, and is defined by
\begin{equation}
\tau:=\intQ\vK\nabla\Psih\cdot\nabla\Psih \dx.
\end{equation}
Note that 
\begin{equation}
\label{eq:defPsih}
\Psih = \dfrac{1}{2}\,(-\phi1+\phi2-\phi3+\phi4).
\end{equation}
The function $\Psih$ is an hourglass mode (see \cite{hourglass}).
The coefficient $\tau$ is the only term in~\eqref{eq:identity}
that depends on the explicit form of the basis functions.
\end{itemize}
\end{theorem}
\begin{proof}
Let $\PZZ$ be the $L^2$ projection onto constants:
\begin{equation*}
\PZZ w := \mintQ w\dx.
\end{equation*}
When the argument is a vector, $\PZZ$ is applied componentwise.
We start from the identity
\begin{equation}
\label{eq:orthogonal-decomposition}
\intQ\vK\nabla u\cdot\nabla v\dx = 
\intQ\vK\PZZ\nabla u\cdot\PZZ\nabla v\dx+
\intQ\vK(\PI-\PZZ)\nabla u\cdot(\PI-\PZZ)\nabla v\dx ,
\end{equation}
which holds true for $u,v\in H^1(\P)$ ({recall}
that $\vK$ is constant on $\Q$).

\subsection{The matrix $\mA$}
%
The matrix $\mA$ of~\eqref{eq:identity} is simply given by the first term
of~\eqref{eq:orthogonal-decomposition}
with $u=\phij$, $v=\phii$:
\begin{equation*}
\mA_{ij} = 
\intQ\vK\PZZ\nabla\phij\cdot\PZZ\nabla\phii\dx
=
\dfrac{1}{|Q|}\,
\vK\left[\intQ\nabla\phij\dx\right]
\cdot
\left[\intQ\nabla\phii\dx\right].
\end{equation*}
Observe that $\mA_{ij}$ is readily computable and does not
depend on the explicit form of the basis functions, since 
by Gauss's formula
\begin{equation}
\label{eq:gradphi}
\intQ\nabla\phii\dx=\intdQ\phii\n\ds=\dfrac{1}{2}\,\dRi ,
\end{equation}
where $\di$ is the vector joining $V_{i-1}$ with $V_{i+1}$
(a diagonal of the quadrilateral, see Fig. \ref{fig:Q}) and $\perp$ denotes clockwise rotation
of $90^\circ$.
Hence,
\begin{equation}
\mA_{ij} = \dfrac{1}{4|\Q|}\,\vK\dRj\cdot\dRi.
\end{equation}
Since in a quadrilateral we have $\d3=-\d1$ and $\d4=-\d2$, the matrix $\mA$ has
a nice block structure:
\begin{equation*}
\mA=\begin{bmatrix}+\mC & -\mC\\ -\mC & +\mC\end{bmatrix}
\quad\text{with}\quad
\mC:=
\begin{bmatrix}
\vK\dR1\cdot\dR1 & \vK\dR1\cdot\dR2\\
\vK\dR1\cdot\dR2 & \vK\dR2\cdot\dR2
\end{bmatrix}.
\end{equation*}
\begin{remark}
Note that if $\vK=\kappa\vI$, then we can remove the rotation:
\begin{equation*}
\mA_{ij} = \dfrac{1}{4|\Q|}\,\kappa\,\dj\cdot\di.
\end{equation*}
\end{remark}

\subsection{The matrix $\mB$}
Now we turn our attention to the second term of \eqref{eq:orthogonal-decomposition}.
The key idea is to write 
\begin{equation*}
\PZZ\nabla w
\quad\text{as}\quad 
\nabla\PN1 w ,
\end{equation*}
where $\PN1 w$ is a projection of $w$ onto linear polynomials.
More precisely, given any function $\vh$ in our local 
space (i.e., a linear combination of the basis functions $\phii$) we want to define a projection
$\PN1\vh$ onto linear polynomials such that:
\begin{itemize}
\item[$\bullet$]
the gradient of $\PN1\vh$ is the $L^2$ projection of the gradient 
of $\vh$, i.e.,
\begin{equation}
\label{eq:defPN1}
\nabla\PN1\vh=\PZZ\nabla\vh=\mintQ\nabla\vh\dx;
\end{equation}
\item[$\bullet$]
$\PN1\vh$ depends only on the value of $\vh$ on the
boundary of $\Q$ (hence it is the same for all generalized barycentric coordinates).
\end{itemize}
We start by noting that \eqref{eq:defPN1} defines the value of the gradient of
$\PN1\vh$, and so it determines $\PN1\vh$ up to a constant:
\begin{equation}
\label{eq:PN}
\PN1\vh = \left(\mintQ\nabla\vh\dx\right)\cdot\vx+\PZt\vh,
\end{equation}
where $\PZt$ is a projection onto constant functions to be fixed.
Now we impose that $\PN1$ is a projection onto linear polynomials, i.e.,
\begin{equation*}
\text{if}\quad\ell(\vx):=\va\cdot\vx + b\quad\text{then}\quad\PN1\ell=\ell.
\end{equation*}
Since $\nabla\ell=\va$, we have
\begin{equation*}
\PN1\ell=\va\cdot\vx + {\widetilde P}_0\ell=\va\cdot\vx + {\widetilde P}_0(\va\cdot\vx + b) = 
\va\cdot\vx + {\widetilde P}_0(\va\cdot\vx) + b = \ell + \va\cdot {\widetilde P}_0\vx.
\end{equation*}
Hence the projection $\PZt$ must satisfy
\begin{equation*}
\PZt\vx=0,\quad\text{i.e.,}\quad \PZt x=0\text{ and }\PZt y=0.
\end{equation*}
A way to impose this condition is to start from an arbitrary
projection onto constants, say $\PZ$, and then define
\begin{equation}
\label{defPZt}
\PZt\vh:=\PZ\vh - \left(\mintQ\nabla\vh\dx\right)\cdot\PZ\vx.
\end{equation}
We end up with the explicit formula
\begin{equation}
\label{explicit-PN1}
\PN1\vh = \left(\mintQ\nabla\vh\dx\right)\cdot(\vx-\PZ\vx)+\PZ\vh,
\end{equation}
where $\PZ$ is an arbitrary projection onto constants.

\begin{remark}
An alternative way to define $\PZt$ from $\PZ$ is to start from \eqref{eq:PN}
and impose the condition
\begin{equation*}
\PZ(\PN1\vh-\vh)=0.
\end{equation*}
In fact, from \eqref{eq:PN} we have
\begin{multline*}
\PZ(\PN1\vh)=
\left(\mintQ\nabla\vh\dx\right)\cdot P_0\vx + P_0{\widetilde P}_0\vh\\
=\left(\mintQ\nabla\vh\dx\right)\cdot P_0\vx + {\widetilde P}_0\vh=P_0\vh
\end{multline*}
from which we get \eqref{defPZt}.
\end{remark}

In order to be able to compute $\PN1\vh$ without actually 
knowing $\vh$ in the interior of $P$, the projector
$\PZ\vh$ must be computable from the boundary values of $\vh$
only. The two most natural choices are:
\begin{itemize}
\item[$\bullet$]
$P_0\vh:=\dfrac{1}{\NV}\sum_{i=1}^{4}\vh(V_i)$
\quad(mean on the vertices of $\Q$).

\newcommand{\Vbar}{\overline{V}}
\newcommand{\Vtilde}{\widetilde{V}}

\medskip\noindent
We have
\begin{equation*}
P_0\vx = \dfrac{1}{4}\sum_{i=1}^{4}\Vi=:\Vbar
\quad\text{(vertex center)}.
\end{equation*}
On taking $\vh=\phii$, recalling \eqref{eq:gradphi} and observing that
\begin{equation*}
\PZ\phii=\dfrac{1}{4},
\end{equation*}
we have
\begin{equation}
\label{PZ-1}
\PN1\phii=\dfrac{1}{2|\Q|}(\vx-\Vbar)\cdot\dRi+\dfrac{1}{4}.
\end{equation}
\item[$\bullet$]
$P_0\vh=\mintdQ\vh\ds$\quad(mean on the boundary of $\Q$).

\medskip\noindent
We have
\begin{equation*}
P_0\vx = 
\dfrac{1}{|\partial\Q|}\sum_{i=1}^{4}
\dfrac{\Vi+\Vip}{2}\,|\ei|=
\dfrac{1}{|\partial\Q|}\sum_{i=1}^{4}
\dfrac{|\eim|+|\ei|}{2}\,\Vi=:\Vtilde.
\end{equation*}
On taking $\vh=\phii$, recalling \eqref{eq:gradphi} and observing that
\begin{equation*}
P_0\phii=\dfrac{1}{|\partial P|}\dfrac{|\eim|+|\ei|}{2},
\end{equation*}
we have
\begin{equation}
\label{PZ-2}
\PN1\phii=\dfrac{1}{2|P|}(\vx-\Vtilde)\cdot\dRi+
\dfrac{1}{|\partial P|}\dfrac{|\eim|+|\ei|}{2}.
\end{equation}
\end{itemize}
It is clear that if all edges have the same length the two definitions
of $\PZ$ coincide. 
In what follows, we will assume that $\PZ$ is defined 
either by \eqref{PZ-1} or by \eqref{PZ-2}.

\begin{remark}
Observe that we cannot take as $\PZ\vh$ the mean value of $\vh$ on $\Q$,
because we want that $\PN1\vh$ depends only on the boundary value of $\vh$
and not on its actual variation inside $\Q$.
\end{remark}

We now show that the function $\Psih$ defined in \eqref{eq:defPsih}
is in the kernel of $\PN1$;
actually, it turns out that $\ker\PN1=\text{span}\,\{\Psih\}$.
Recall that the function $\Psih$ is defined as the (unique) function
in our local space such that $\Psih(\Vi)=\dfrac{(-1)^i}{2}$.
We observe that $\Psih$ has zero mean value on each edge, and hence
\begin{equation*}
\intQ\nabla\Psih\dx=\intdQ\Psih\n\ds=
\sum_{i=1}^4\left[\int_{\ei}\Psih\ds\right]\ni=0.
\end{equation*}
Furthermore, $P_0(\Psih)=0$
for any of the two choices of $P_0$ given above.
Hence, from~\eqref{explicit-PN1}, we have $\PN1\Psih=0$.


\newcommand{\umz}{\frac{1}{2}}

The function $\Psih$ is linearly independent of the standard first-degree monomials 
$\{1,x,y\}$; hence the four functions $\{1,x,y,\Psih\}$ are a basis
for our local space so that in particular any $\phii$ can be written as a linear 
combination of $1,x,y,\Psih$. To find the coefficients, we exploit 
the fact that the $\phii$ are generalized
barycentric coordinates, that is
\begin{equation*}
\sum_{i=1}^4\phii=1,\quad
\sum_{i=1}^4 x_i\phii=x,\quad
\sum_{i=1}^4 y_i\phii=y
\end{equation*}
plus the equation defining $\Psih$:
\begin{equation*}
\Psih = \dfrac{1}{2}\,(-\phi1+\phi2-\phi3+\phi4).
\end{equation*}
In matrix form, we have
\begin{equation}
\label{transf}
\begin{bmatrix}
1     &    1  &    1  &    1  \\
\x1   &  \x2  &  \x3  &  \x4  \\
\y1   &  \y2  &  \y3  &  \y4  \\
-\umz & \umz  & -\umz & \umz 
\end{bmatrix}
\begin{bmatrix}
\phi1 \\ \phi2 \\ \phi3 \\ \phi4
\end{bmatrix}
=
\begin{bmatrix}
1 \\ x \\ y \\ \Psih
\end{bmatrix}.
\end{equation}
Let us denote by $T_i$ the signed area of the triangle obtained by removing
the vertex $\Vi$ from the quadrilateral $\Q$ and joining vertex $V_{i-1}$ with
vertex $V_{i+1}$ (see Fig. \ref{fig:Ti}).
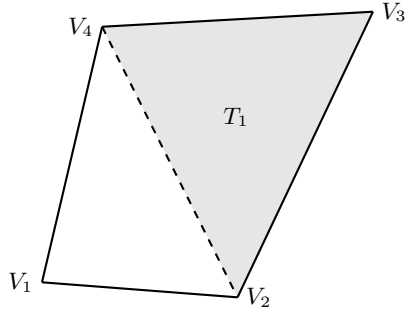
\begin{figure}[ht]
\begin{center}
\begin{tikzpicture}[scale=2]
\coordinate (v1) at (0.1,0.1);
\coordinate (v2) at (-0.5,-1);
\coordinate (v3) at (2,-0.7);
\coordinate (v4) at (2.2,0.3);
\coordinate (v5) at (1,2);
\coordinate (v6) at (-0.8,1.9);
\coordinate (v7) at (-1.2,0.2);
\coordinate (v67) at ($0.5*(v6)+0.5*(v7)$);
\coordinate (v71) at ($0.5*(v7)+0.5*(v1)$);
\coordinate (v16) at ($0.5*(v1)+0.5*(v6)$);
\fill [fill=gray!20] (v1) -- (v5) -- (v6) -- cycle;
\draw [thick]  (v1) -- (v5);
\draw [thick]  (v5) -- (v6);
\draw [thick]  (v6) -- (v7);
\draw [thick]  (v7) -- (v1);
\node [right] at (v1) {$V_2$};
\node [left] at (v6) {$V_4$};
\node [left] at (v7) {$V_1$};
\node [right] at (v5) {$V_3$};
\draw [thick] [dashed] (v6) -- (v1);
\node [right] at ($(v16)+(0.3,0.3)$) {$T_1$};
\tikzAngleOfLine(v6)(v7){\angle};
%
\tikzAngleOfLine(v7)(v1){\angle};
%
%
\end{tikzpicture}
\end{center}
\caption{Signed area $T_1$.}\label{fig:Ti}
\end{figure}
\noindent
Therefore,
\begin{equation*}
\begin{split}
T_1 = \dfrac{1}{2}\,\det\begin{bmatrix}1&1&1\\ x_2&x_3&x_4\\y_2&y_3&y_4\end{bmatrix},\
T_2 = \dfrac{1}{2}\,\det\begin{bmatrix}1&1&1\\ x_1&x_3&x_4\\y_1&y_3&y_4\end{bmatrix},\\
T_3 = \dfrac{1}{2}\,\det\begin{bmatrix}1&1&1\\ x_1&x_2&x_4\\y_1&y_2&y_4\end{bmatrix},\ 
T_4 = \dfrac{1}{2}\,\det\begin{bmatrix}1&1&1\\ x_1&x_2&x_3\\y_1&y_2&y_3\end{bmatrix}.
\end{split}
\end{equation*}
Now, let us consider the coefficient matrix of the linear system
in~\eqref{transf}.
Carrying out the expansion with respect to the last row, it can be
directly verified that its determinant is equal to $2|\Q|$. By directly solving
system \eqref{transf} through Cramer's rule, and expanding the determinant
with respect to the $i$-th column, we obtain
\begin{equation*}
\phii = (a_i+b_i x+c_i y) + (-1)^i\,\dfrac{T_i}{|\Q|}\Psih,\quad i=1,\dots,4,
\end{equation*}
where
\begin{align*}
a_i&=\dfrac{(-1)^iT_i+(x_{i+1}y_{i-1}-x_{i-1}y_{i+1})}{2\Q},\\
b_i&=\dfrac{y_{i+1}-y_{i-1}}{2|\Q|}=\dfrac{(\di)_y}{2|\Q|},\quad
c_i=-\dfrac{x_{i+1}-x_{i-1}}{2|\Q|}=-\dfrac{(\di)_x}{2|\Q|}.
\end{align*}
We conclude that
\begin{equation*}
\phii(\vx)=a_i+\dfrac{\dRi}{2|\Q|}\cdot\vx
+ (-1)^i\,\dfrac{T_i}{|\Q|}\Psih(\vx).
\end{equation*}
On defining the non-dimensional quantities 
\begin{equation*}
T_i':=(-1)^i \dfrac{T_i}{|\Q|},
\end{equation*}
we have
\begin{equation*}
\phii(\vx)=a_i+\dfrac{\dRi}{2|\Q|}\cdot\vx
+ {T_i'}\,\Psih(\vx).
\end{equation*}
Hence we have the following identities:
\begin{equation}
\label{eq:I-PN}
(\PI-\PN1)\phii={T_i'}\,\Psih,
\end{equation}
and
\begin{equation}
(\PI-\PZZ)\nabla\phii=\nabla\phii-\nabla\PN1\phii=
\nabla(\phii-\PN1\phii)={T_i'}\,\nabla\Psih.
\end{equation}
The second term of~\eqref{eq:orthogonal-decomposition} can then be
written as
\begin{equation*}
\intQ\vK(\PI-\PZZ)\nabla\phij\cdot(\PI-\PZZ)\nabla\phii\dx
=
{T_i'\, T_j'} \intQ\vK\nabla\Psih\cdot\nabla\Psih\dx
\end{equation*}
and the matrix $\mB$ of \eqref{eq:identity} is identified by
\begin{equation*}
\mB_{ij}={T_i'\, T_j'}.
\end{equation*}
Setting $\vgamma := 
\begin{bmatrix}
T_1' & T_2' & T_3' &T_4'
\end{bmatrix}'$,
the matrix $\mB$ can be written as
$\mB=\vgamma\vgamma^\text{T}$.
Note that if $\Q$ is a parallelogram, we have
\begin{equation*}
T_i'=(-1)^i\,\dfrac{1}{2}
\end{equation*}
so that $\mB$ is independent on $\Q$:
\begin{equation*}
\mB=\dfrac{1}{4}\,
\begin{bmatrix}
+1 & -1 & +1 & -1 \\
-1 & +1 & -1 & +1 \\
+1 & -1 & +1 & -1 \\
-1 & +1 & -1 & +1 \\
\end{bmatrix}.
\end{equation*}
We end up with the formula:
\begin{equation*}
\intQ\vK\nabla\phij\cdot\nabla\phii\dx =
\dfrac{1}{4|\Q|}\,\vK\dRj\cdot\dRi+
\left[\intQ\vK\nabla\Psih\cdot\nabla\Psih \dx\right]
\vgamma\vgamma^\text{T}
= \mA+\tau\mB.
\end{equation*}
Clearly, the only term that depends on the variation
of the generalized barycentric coordinates inside $\Q$ is the
coefficient 
\begin{equation*}
\tau=\intQ\vK\nabla\Psih\cdot\nabla\Psih \dx,
\end{equation*}
which is the \textit{energy of the hourglass mode $\Psih$}.
\end{proof}

\section{Value of $\tau$ for some GBCs} 

For a general quadrilateral usually
the value of $\tau$ does not
have an expression in closed form.
In this Section, we report the value of $\tau$
for isoparametric finite element shape functions and 
Wachspress
coordinates over rectangles and parallelograms.

\subsection{Rectangles}
For a rectangle, isoparametric FEM, 
harmonic generalized barycentric coordinates~\cite{Joshi:2007:HCF}
and Wachspress shape functions
coincide. If $\Q=[0,a]\times[0,b]$, 
the value of $\tau$ 
for a general $\vK$ is given by
\begin{equation}
\label{eq:taurec}
\tau = \dfrac{b^2\kappa_{11}+a^2\kappa_{22}}{3ab}.
\end{equation}

\subsection{Parallelograms}

For a parallelogram, isoparametric FEM and Wachspress GBCs coincide. In this case, the value of $\tau$ is given by:
\begin{itemize}
\item[$\bullet$]
When $\vK=\kappa\vI$, and 
$\Q$ is a parallelogram of sides $a$ and $b$
with angle $\theta$:
\begin{equation*}
\tau = \kappa\,\dfrac{a^2+b^2}{3ab\sin\theta}.
\end{equation*}
\item[$\bullet$]
For general $\vK$, and if $\Q$ is a 
parallelogram of sides $a$ and $b$
with angle $\theta$, with side $a$ parallel to the $x$-axis:
\begin{equation*}
\tau=
\dfrac{a^2\kappa_{22}+b^2\kappa_{11}}{3ab\sin\theta}
+\frac{b\,\left(
{(\kappa_{22}-\kappa_{11}){\,\cos^2\theta}}
-{2\,\kappa_{12}\,\cos\theta\,\sin\theta}\right)}{3a\sin\theta}.
\end{equation*}
\end{itemize}

\section{Connection with Virtual Element Method}
The virtual element method is a fairly recent methodology
that in particular extends classical 
finite elements to
polygonal and polyhedral meshes, see~\cite{ACTA-VEM} and 
the references therein. The keys ideas in VEM are:
\begin{itemize}
\item[$\bullet$]
the local space is \say{virtual} in the sense that functions are known
only through their degrees of freedom;
\item[$\bullet$]
the element stiffness matrix is split into a \say{consistency} term that takes
care of the accuracy plus a \say{stability} term that ensures
stability without violating consistency: 
\begin{equation*}
\mKVEM = \mKc + \mKs.
\end{equation*}
\end{itemize}

\subsection{Linear virtual element on a polygon}

In the case of linear virtual elements for the
diffusion equation, 
the \say{basic} local space coincides with harmonic GBCs,
whereas 
the \say{enhanced} version~\cite{projectors} is still a
GBC, but the local functions 
are no longer harmonic.
The consistency and stability matrices in the VEM
are built upon 
the construction of the $\PN1$ projection, which can be extended
to a general polygon $\P$ with $\NV$ vertices by following
the above construction, 
and is still given by~\eqref{explicit-PN1} 
(see \cite{volley,projectors,hitchhikers} for the details).

In the case $\vK$ is a constant matrix,
for the diffusion problem the local 
VEM consistency matrix $\mKc$ concides with the matrix $\mA$
of the decomposition~\eqref{eq:identity}:
\begin{equation*}
(\mKc)_{ij} =
\mA_{ij} = 
\intP\vK\PZZ\nabla\phij\cdot\PZZ\nabla\phij\dx =
\intP\vK\nabla\PN1\phij\cdot\nabla\PN1\phij\dx =
\dfrac{1}{4|\P|}\,\vK\dRj\cdot\dRi.
\end{equation*}

The VEM stability matrix $\mKs$ for a polygon $\P$
is built in the following way.
Let $\calS$ be a symmetric bilinear form defined on the
local space that \say{scales} on the kernel of $\PN1$
like the bilinear form associated with the differential equation, 
i.e., there exist two constants 
$\alpha_*$ and $\alpha^*$ 
independent
of the element $\P$
such that
\begin{equation*}
\alpha_* \intP\vK\nabla\vh\cdot\nabla\vh\dx 
\leq \calS(\vh,\vh) \leq
\alpha^* \intP\vK\nabla\vh\cdot\nabla\vh\dx 
\quad\text{for all }\vh\in\ker\PN1.
\end{equation*}
Then if we define
\begin{equation}
(\mKs)_{ij} := \calS\left((\PI-\PN1)\phij,(\PI-\PN1)\phii\right) ,
\end{equation}
we have a convergent method (see Theorem 3.1 and 4.1 
of~\cite{volley}). 

In order to construct a computable $\calS$
satisfying the hypotheses above,
we proceed in the following way.
Define $\dof_i(\vh)$ as the $i$-th degree of freedom of $\vh$
in the linear case, that is
\begin{equation*}
\dof_i(\vh):=\vh(\Vi),\quad i=1,\dots,\NV 
\end{equation*}
and then set
\begin{equation*}
\calS(\uh,\vh):=\tauVEM\sum_{i=1}^\NV\dof_i(\uh)\,\dof_i(\vh)
\quad\text{(dofi-dofi stabilization)}
\end{equation*}
where $\tauVEM$ is a parameter to be fixed
(see Section 4.2 of~\cite{variable-coefficients}). 
Under reasonable assumptions on the mesh sequence
(quadrilaterals are not degenerate) 
to have convergence when the mesh size goes to zero,
we can take any non-zero constant for $\tauVEM$,
provided that all $\tauVEM$'s for all polygons and for
all meshes are uniformly bounded from below and above.
In other words, we can say that \textit{$\tauVEM$ must 
scale like 1}. If we are in three dimensions we 
require that \textit{$\tauVEM$ scales like $h$}.

Hence, the final expression for the local stability
$\mKs$ is the following:
\begin{equation*}
(\mKs)_{ij} = \tauVEM\,
\sum_{k=1}^\NV
\dof_k[(\PI-\PN1)\phij]\,
\dof_k[(\PI-\PN1)\phii].
\end{equation*}
In the next section, we examine a
practical choice for $\tauVEM$.

\subsection{Revisiting quadrilaterals}
From now on we consider a polygon to be a 
quadrilateral $\Q$.
In this case, we have
\begin{equation*}
\sum_{k=1}^4
\dof_k[(\PI-\PN1)\phij]\,
\dof_k[(\PI-\PN1)\phii]
=
\mB_{ij},
\end{equation*}
where $\mB$ is the matrix appearing in the 
decomposition~\eqref{eq:identity}.
In fact, since 
$(\PI-\PN1)\phii = T_i'\,\Psih$
(see~\eqref{eq:I-PN})
and
$\Psih=\dfrac12\sum_{\ell=1}^{4}(-1)^\ell\phik$,
we have
\begin{equation*}
\dof_k[(\PI-\PN1)\phii]=
T_i'\sum_{\ell=1}^4\dfrac{(-1)^\ell}{2}\dof_k(\phil)=
T_i'\sum_{\ell=1}^4\dfrac{(-1)^\ell}{2}\delta_{kl}=
T_i'\dfrac{(-1)^k}{2}
\end{equation*}
so that
\begin{equation*}
\sum_{k=1}^4\dof_k[(\PI-\PN1)\phij]\,\dof_k[(\PI-\PN1)\phii]=
T_j'\,T_i'\sum_{k=1}^{4}\left[\dfrac{(-1)^k}{2}\right]^2=
T_j'\,T_i'=\mB_{ij}.
\end{equation*}
To summarize, the element VEM stiffness matrix
for a quadrilateral can be written as
\begin{equation}
\label{eq:identity-VEM}
\mKVEM = \mA + \tauVEM\mB.
\end{equation}

\section{How to choose $\tauVEM$?} 

The decompositions in~\eqref{eq:identity} and
in~\eqref{eq:identity-VEM} for GBCs and VEM, respectively are: 
\begin{equation*}
\mK = \mA + \tau\mB
\quad\text{and}\quad
\mKVEM = \mA + \tauVEM\mB,
\end{equation*}
which are formally equal, but very different in practice. 
\begin{itemize}
\item[$\bullet$]
In the case of GBCs,
the value of the parameter $\tau$ is well defined as the energy of a
particular function (the hourglass mode $\Psih$) of the local space:
\[
\tau=\intQ\vK\nabla\Psih\cdot\nabla\Psih\dx.
\]
The value of $\tau$ can be computed by quadrature;
actually, computing $\tau$ by a quadrature formula and 
using the 
decomposition~\eqref{eq:identity} is equivalent 
to approximating directly $\mK$ 
(as defined in~\eqref{eq:stiffness-matrix}) 
with the same quadrature formula.
\item[$\bullet$]
For VEM, the parameter $\tauVEM$ is left unspecified, and we only require
that convergence hypotheses are satisfied. In this
very particular case (quadrilaterals, linear VEM, $\vK$ constant) 
the value of $\tauVEM$ could be in principle 
identified with the energy of the corresponding 
houglass mode of the VEM space, but the extremely wide range of applicability of VEM
(general polygons and polyhedra, polynomials of any order, 
elasticity, Navier--Stokes, magnetostatics problems) prevents
in most instances a constructive approach for the computation of $\tauVEM$.
\end{itemize}
Here we will use the interpretation of the \say{correct} $\tauVEM$ 
as the energy of the hourglass mode $\Psih$ to draw some general
conclusions about the design of $\tauVEM$ and the consequences 
of having employed an \say{incorrect} $\tauVEM$.

First of all, it seems natural to include in $\tauVEM$ some information from $\vK$. This is not needed for convergence, but for a given
mesh if $\vK$ is large 
we could have a 
marked difference in the solution. 
We then set:
\begin{equation}
\label{eq:tauVEM}
\tauVEM = \dfrac{\trace\vK}{2}.
\end{equation}
Hence, regardless of the shape of the quadrilateral $\Q$
and of the local VEM space, in the VEM the 
energy of the hourglass mode $\Psih$ is always set to 
$\trace\vK / 2$.

As observed in many papers on VEM, the sensitivity with respect to
the value of $\tauVEM$ is usually mild: the value 
in~\eqref{eq:tauVEM}
works well for a wide range of polygonal shapes.
%
%
We now provide explanation and clarification for this observation.

\subsection{The worst case for VEM}
The worst scenario for VEM on a given mesh is the following:
\begin{itemize}
\item[$\bullet$]
the value of $\tauVEM$ is a bad approximation of
the energy of $\Psih$;
\item[$\bullet$] 
the exact solution of the PDE \say{contains} the hourglass mode $\Psih$.
\end{itemize}

Note, however, that the hourglass mode $\Psih$ depends on the mesh;
hence upon refinement we are led to a \say{better} solution.

\subsection{Effect of $\tauVEM$ in a Laplace problem}
Consider the following Laplace problem with inhomogeneous
Dirichlet boundary conditions:
\begin{equation}
\label{eq:uexact}
\left\{
\begin{aligned}
-\Delta u &= 0\quad\text{ in } \Omega = (0,1)^2, \\
        u &= g\quad\text{ on } \partial\Omega ,
\end{aligned}
\right.
\end{equation}
where $g$ is a continuous, piecewise linear function that
oscillates 20 times on each edge from $-1/4$ to $+1/4$. 
The exact solution $u$ decays very quickly to zero 
inside the domain; see the reference solution shown in Fig.~\ref{fig:ue}.
\begin{figure}[ht]
    \centering
    {\includegraphics[trim = 1.6cm 1cm 1.6cm 1.8cm, clip, scale=0.6]{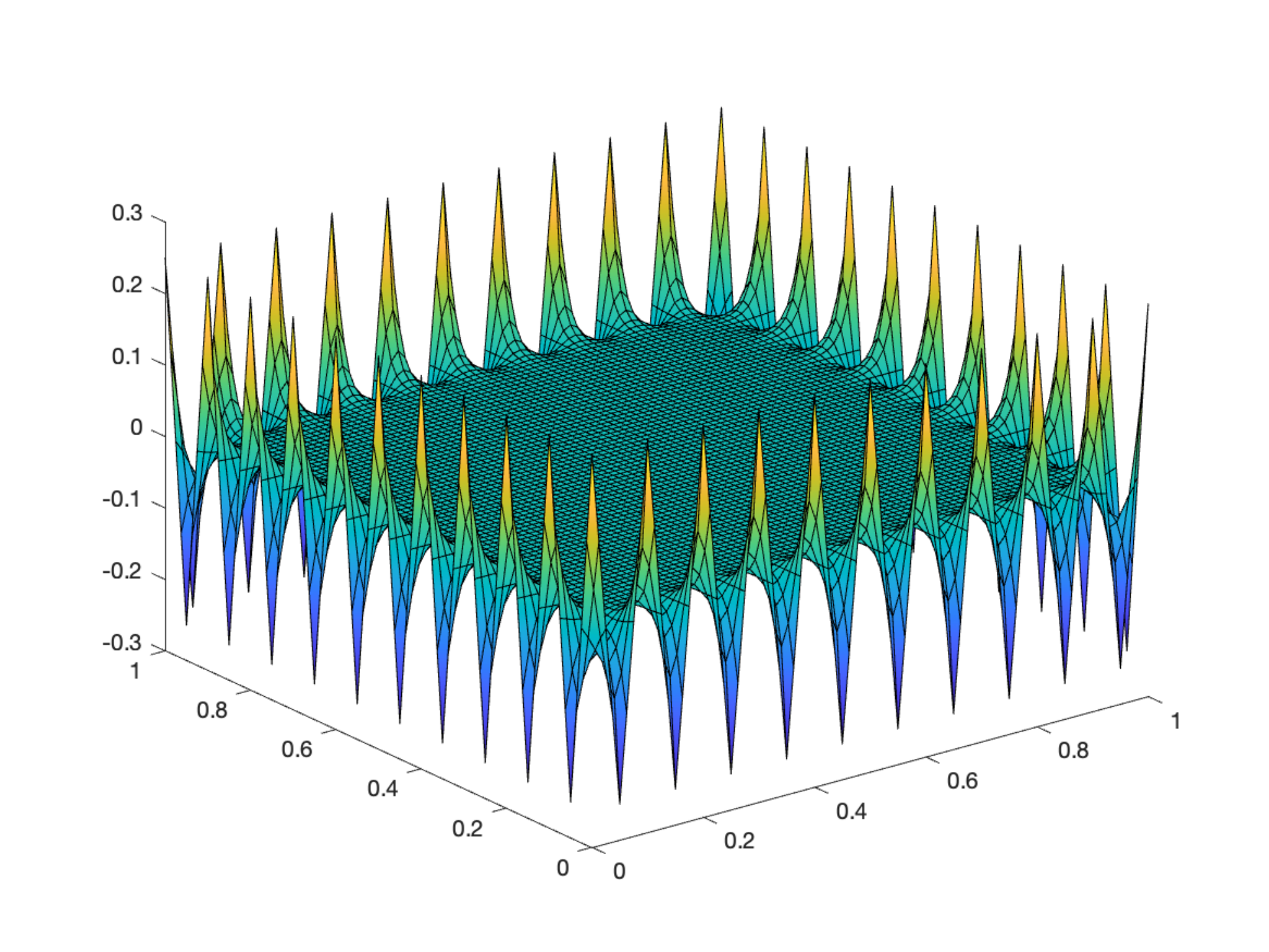}}
    \caption{Reference finite element solution to
             the problem posed in~\eqref{eq:uexact}.}
    \label{fig:ue}
\end{figure}

Now we want to use a mesh such that the hourglass mode $\Psih$ has a 
strong component in the exact solution $u$. We divide 
$\Omega$
into $20\times20$ uniform squares, in such a way that the boundary condition
$g$ oscillates precisely as $\Psih$. 
Hence, if the numerical scheme does not
adopt the correct value of the energy of $\Psih$, it will propagate
$g$ inside the domain and lead to an incorrect solution.
This is similar to hourglass modes in a FEM mesh
that can become
communicable and wreck the solution~\cite{TB:book}.
In the first experiment we consider
isoparametric FEM
(see Fig.~\ref{fig:20x20-tau=2over3}) and
VEM with $\tauVEM=1$ (see Fig.~\ref{fig:20x20-tau=1}).
Given that the mesh is coarse, the numerical solutions 
are adequate; however, note that the two solutions are
distinct.
\begin{remark}
From \eqref{eq:tauVEM},
$\tauVEM=\trace\vK/2$;
here $\vK=\vI$, so $\tauVEM=1$. 
On a square,
VEM with $\tauVEM=2/3$
is identical to 
isoparametric FEM
(see \eqref{eq:taurec}).
\end{remark}
\begin{figure}[!htb]
   \begin{subfigure}{0.48\textwidth}
     \centering
   {\includegraphics[trim = 1.6cm 1cm 1.6cm 1.8cm, clip, scale=0.4]{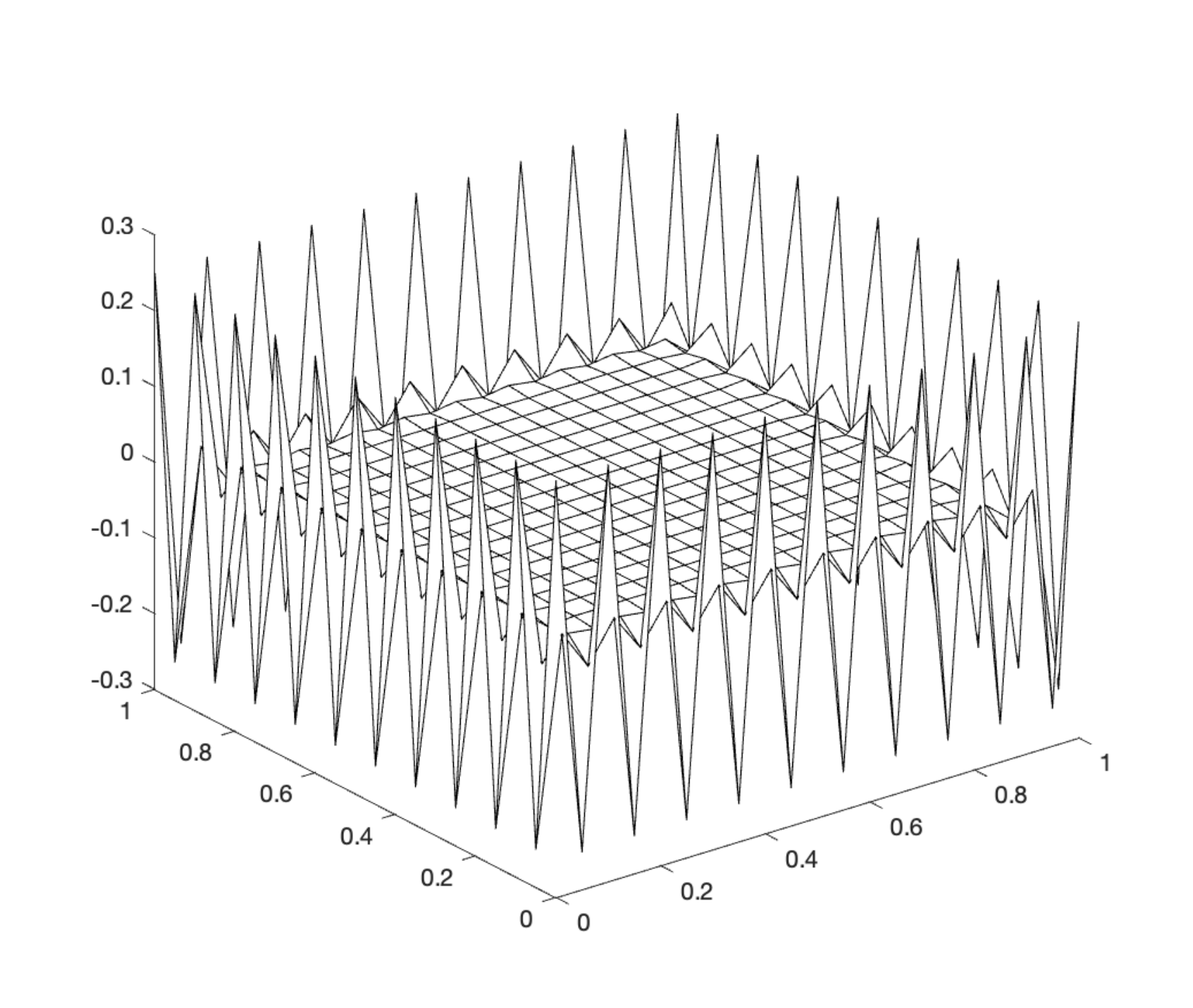}}
    \subcaption{Isoparametric FEM}
    \label{fig:20x20-tau=2over3}
   \end{subfigure}\hfill
   \begin{subfigure}{0.48\textwidth}
     \centering
    {\includegraphics[trim = 1.6cm 1cm 1.6cm 1.8cm, clip, scale=0.4]{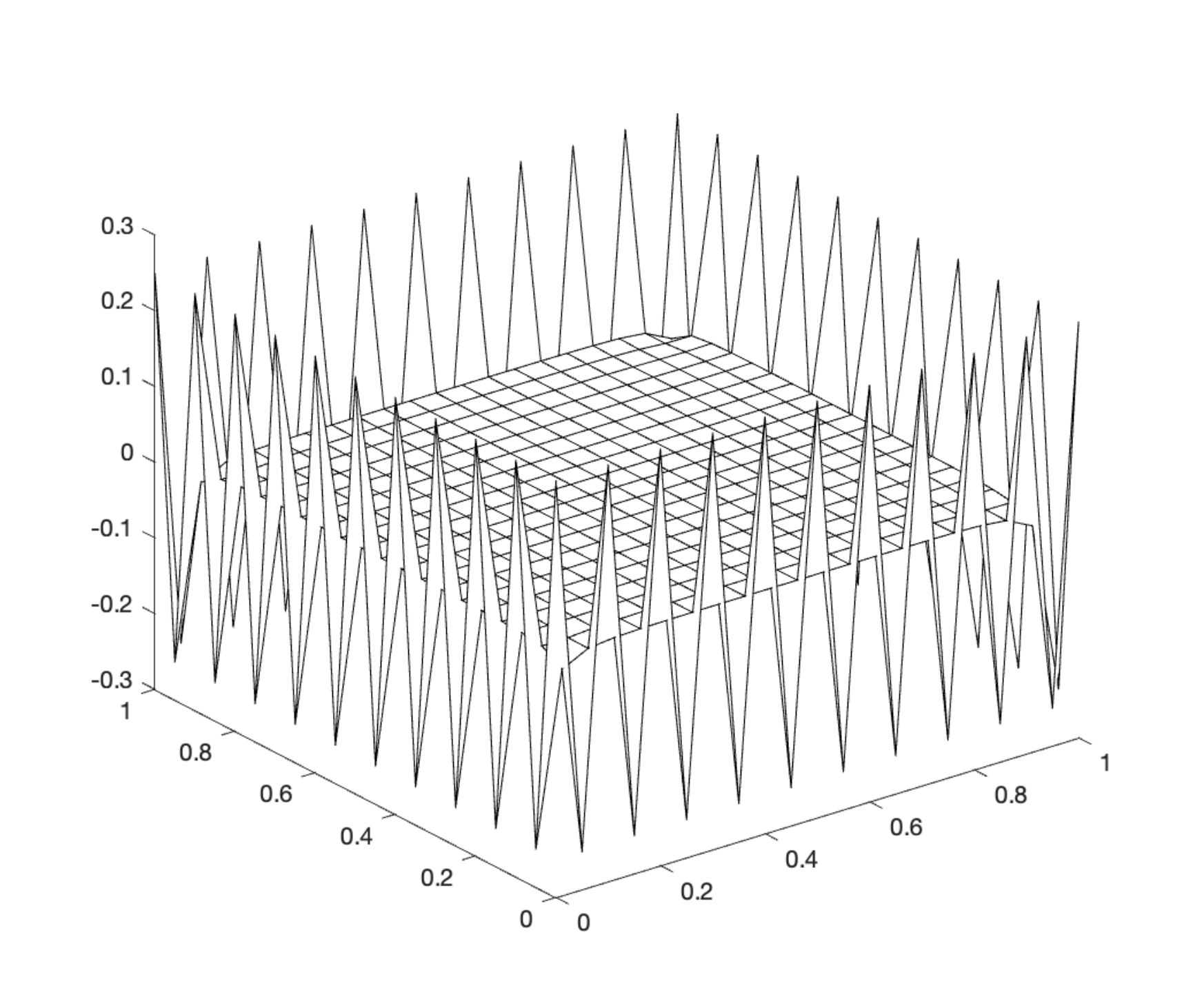}}
    \subcaption{VEM ($\tauVEM = 1$)}
    \label{fig:20x20-tau=1}
   \end{subfigure}
   \caption{Numerical solution for isoparametric FEM and 
            VEM on a $20 \times 20$ mesh.}
            \label{fig:20x20-tau-FEM-VEM}
\end{figure}

Now we take VEM with $\tauVEM=0.1$, and we can see that 
the virtual element solution worsens
(see Fig.~\ref{fig:20x20-tau=0.1}).
If we further decrease the value of $\tauVEM$ to $0.01$, 
then the boundary data $g$ is almost free to move in the domain since dissipation is very small
(Fig.~\ref{fig:20x20-tau=0.01}).
The numerical solution is also unstable
for $\tauVEM=10$ and 
$\tauVEM=100$, as is observed in
Figs.~\ref{fig:20x20-tau=10} and~\ref{fig:20x20-tau=100}.
\begin{figure}[!htb]
   \begin{subfigure}{0.48\textwidth}
     \centering
    {\includegraphics[trim = 1.6cm 1cm 1.6cm 1.8cm, clip, scale=0.4]{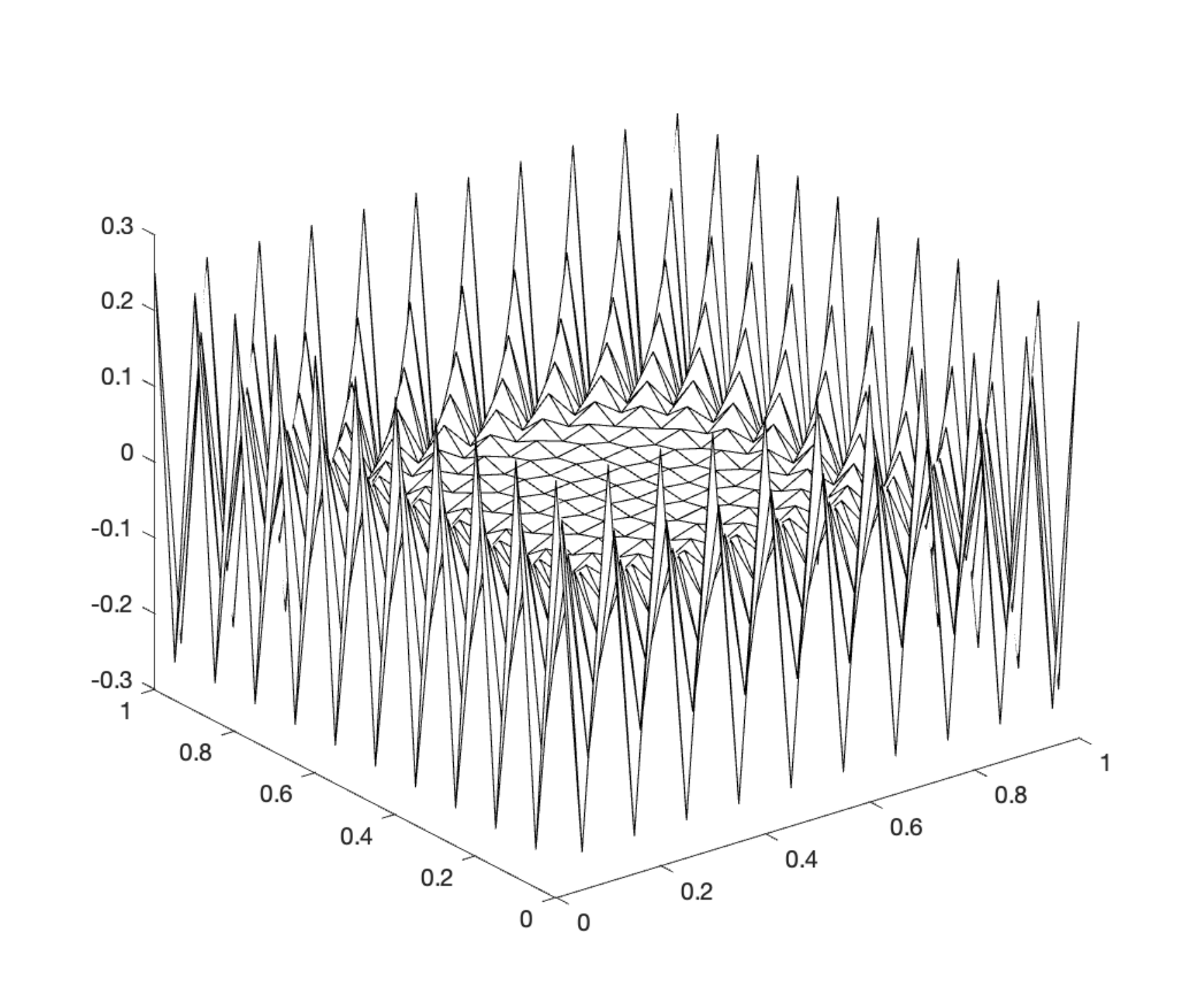}}
    \subcaption{$\tauVEM=0.1$}
    \label{fig:20x20-tau=0.1}
   \end{subfigure}\hfill
   \begin{subfigure}{0.48\textwidth}
     \centering
    {\includegraphics[trim = 1.6cm 1cm 1.6cm 1.8cm, clip, scale=0.4]{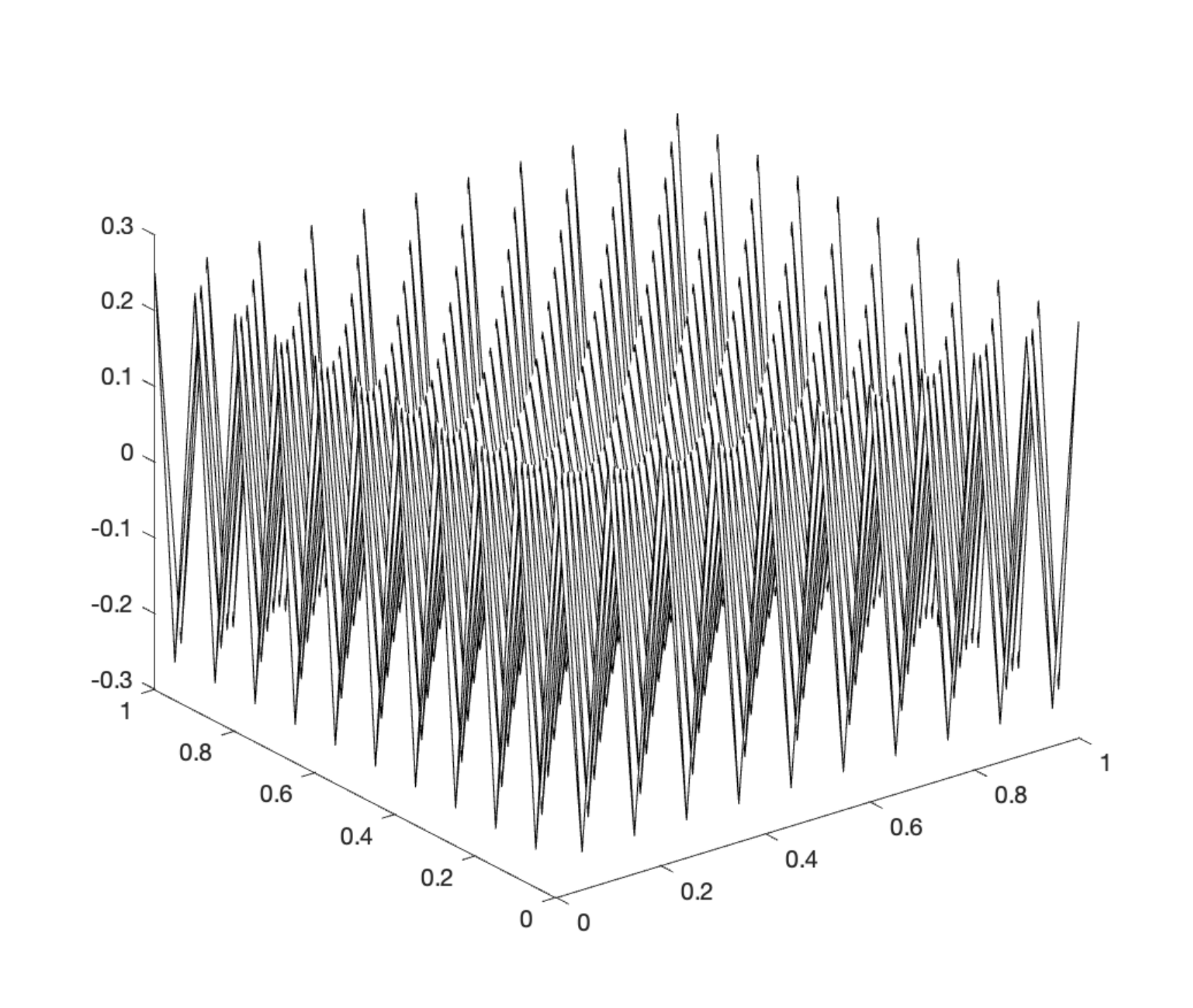}}
    \subcaption{$\tauVEM=0.01$}
    \label{fig:20x20-tau=0.01}
   \end{subfigure}
   \begin{subfigure}{0.48\textwidth}
     \centering
    {\includegraphics[trim = 1.6cm 1cm 1.6cm 1.8cm, clip, scale=0.4]{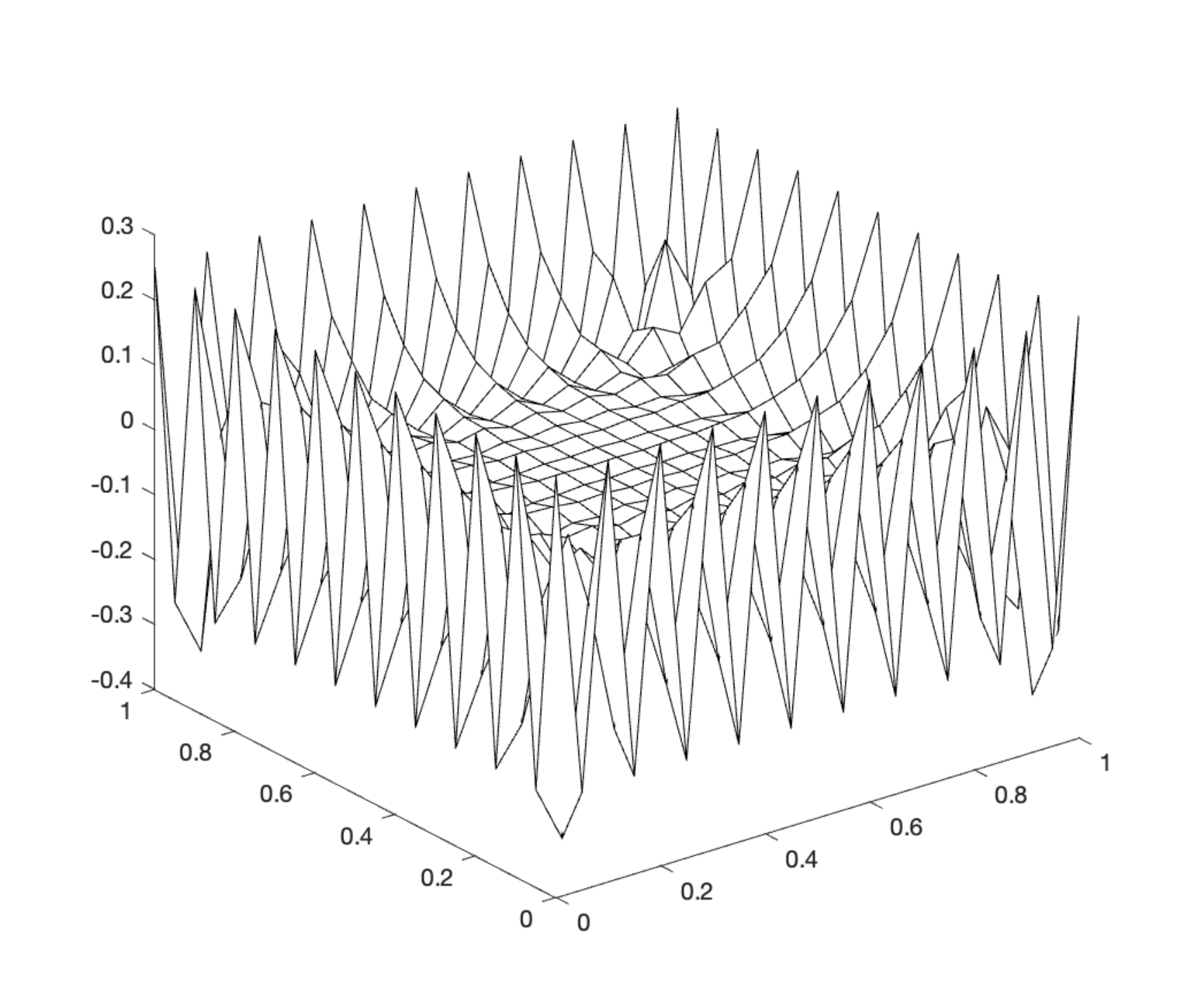}}
    \subcaption{$\tauVEM=10$}
    \label{fig:20x20-tau=10}
   \end{subfigure}\hfill
   \begin{subfigure}{0.48\textwidth}
     \centering
    {\includegraphics[trim = 1.6cm 1cm 1.6cm 1.8cm, clip, scale=0.4]{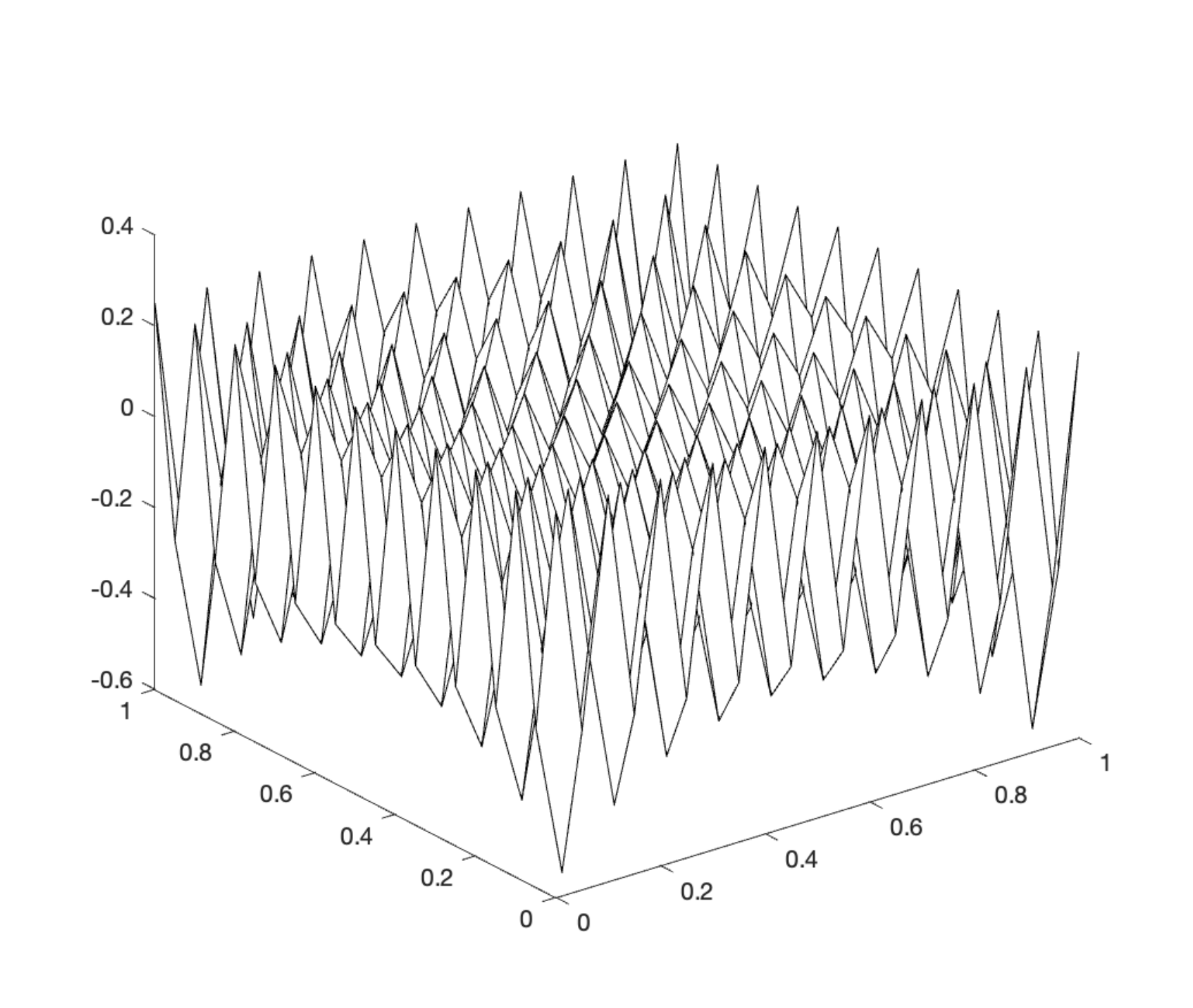}}
    \subcaption{$\tauVEM=100$}
    \label{fig:20x20-tau=100}
   \end{subfigure}
    \caption{Numerical solution using different $\tauVEM$ for
            the VEM on a $20 \times 20$ mesh.}
            \label{fig:20x20-tau-VEM}
\end{figure}

However, if we refine and consider a $40\times40$ mesh of
uniform squares,
the two values $0.1$ and $10$ for 
$\tauVEM$ become acceptable, as is seen in 
Figs.~\ref{fig:40x40-tau=0.1} 
and~\ref{fig:40x40-tau=10}.
The reason for this observation
is that now the hourglass mode $\Psih$ has a smaller
component in the exact solution $u$, and even if its energy
of $\Psih$ is not precise, its impact on the numerical solution is not severe. As the mesh is further refined, the virtual
element solution becomes more and more 
insensitive to $\tauVEM$, as 
shown in Figs.~\ref{fig:80x80-tau=0.01}--\ref{fig:80x80-tau=100}.
\begin{figure}[!htb]
   \begin{subfigure}{0.48\textwidth}
     \centering
    {\includegraphics[trim = 1.6cm 1cm 1.6cm 1.8cm, clip, scale=0.4]{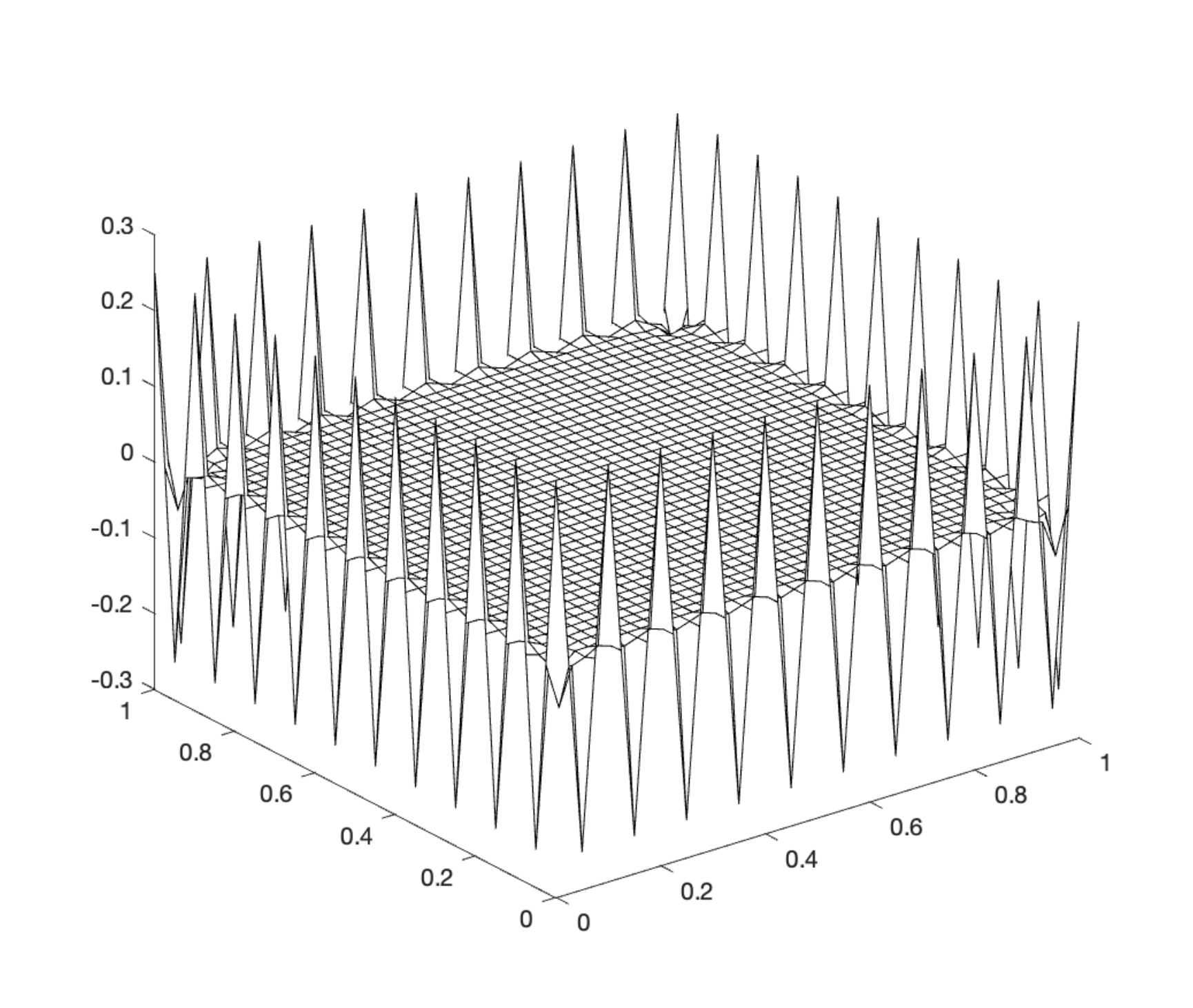}}
    \caption{$\tauVEM=0.1$}
    \label{fig:40x40-tau=0.1}
   \end{subfigure}\hfill
   \begin{subfigure}{0.48\textwidth}
     \centering
    {\includegraphics[trim = 1.6cm 1cm 1.6cm 1.8cm, clip, scale=0.4]{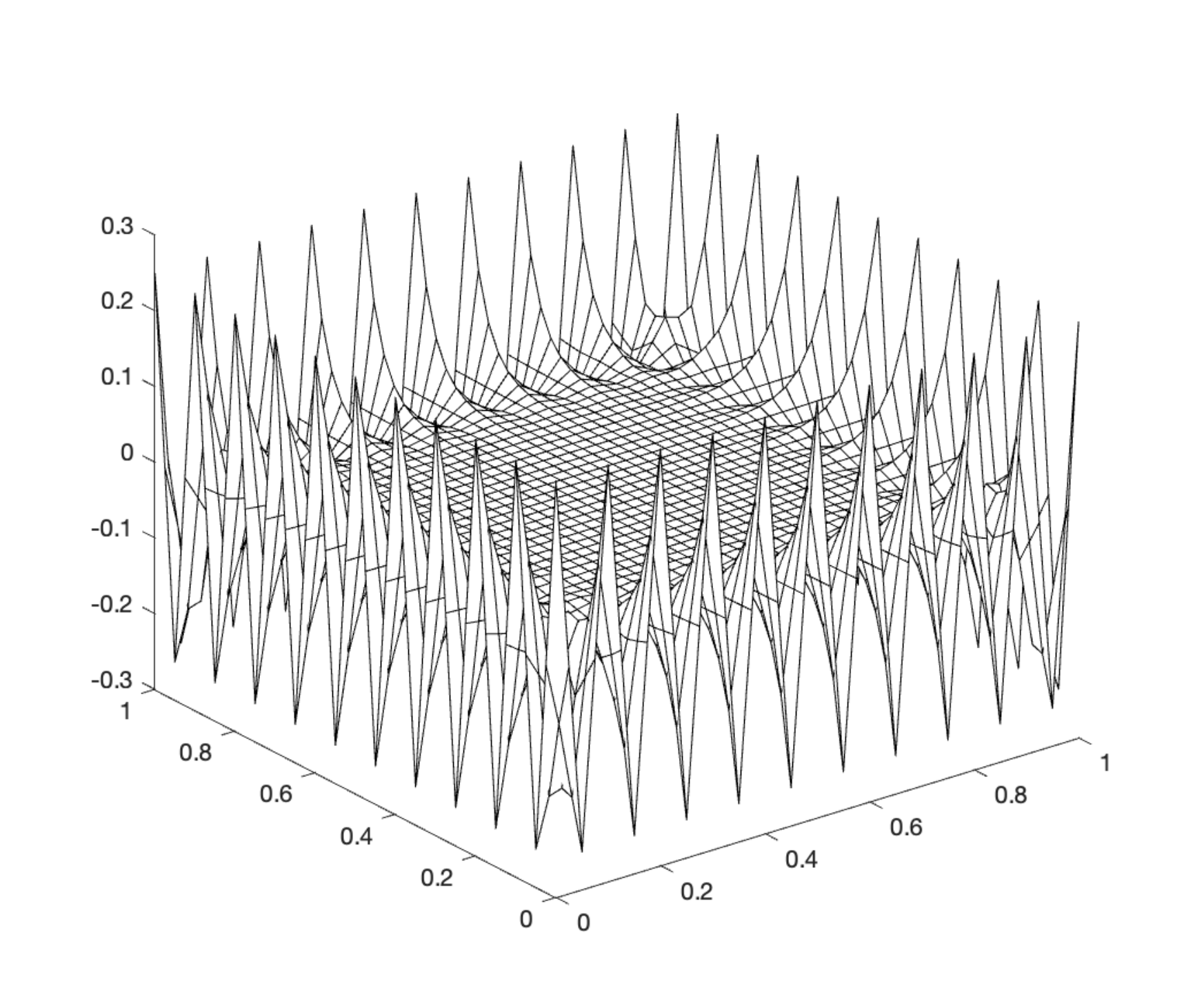}}
    \caption{$\tauVEM=10$}
    \label{fig:40x40-tau=10}
   \end{subfigure}
   \caption{Numerical solution using different $\tauVEM$ for
            the VEM on a $40 \times 40$ mesh.}
            \label{fig:40x40-tau-VEM}
\end{figure}
\begin{figure}[!htb]
   \begin{subfigure}{0.48\textwidth}
     \centering
    {\includegraphics[trim = 1.6cm 1cm 1.6cm 1.8cm, clip, scale=0.4]{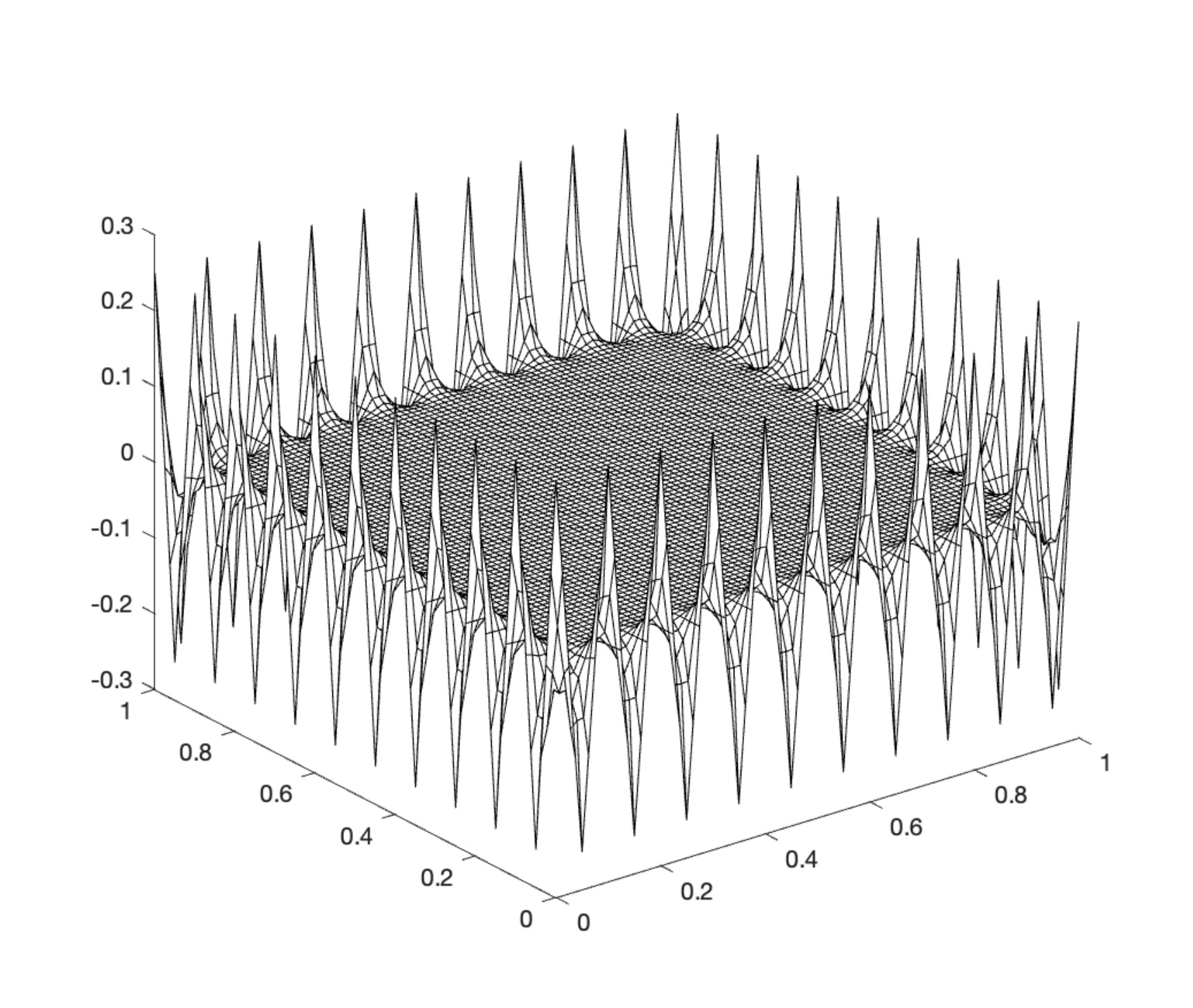}}
    \caption{$\tauVEM=0.01$}
    \label{fig:80x80-tau=0.01}
   \end{subfigure}\hfill
   \begin{subfigure}{0.48\textwidth}
     \centering
    {\includegraphics[trim = 1.6cm 1cm 1.6cm 1.8cm, clip, scale=0.4]{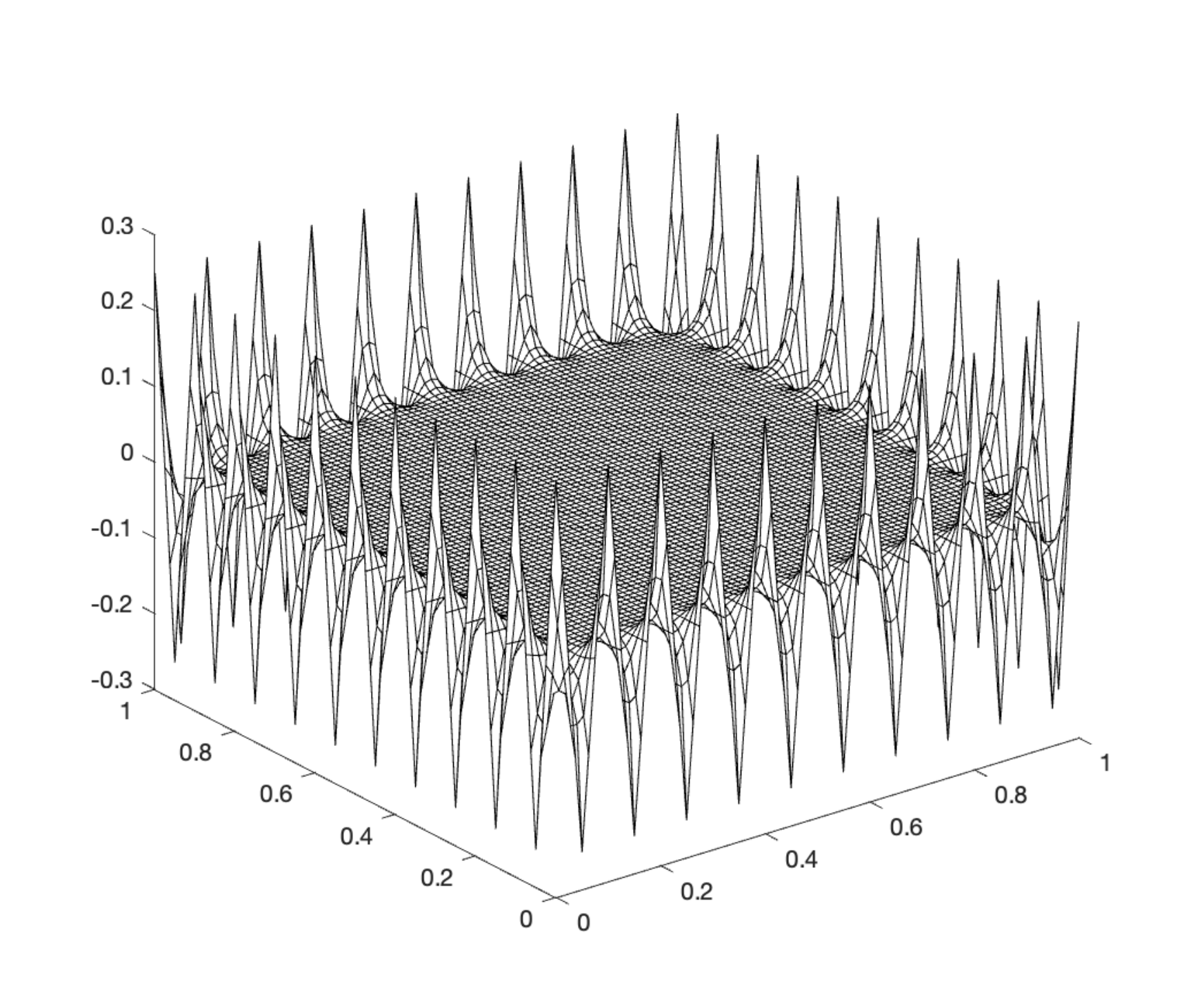}}
    \caption{$\tauVEM=0.1$}
    \label{fig:80x80-tau=0.1}
   \end{subfigure}
   \begin{subfigure}{0.48\textwidth}
     \centering
    {\includegraphics[trim = 1.6cm 1cm 1.6cm 1.8cm, clip, scale=0.4]{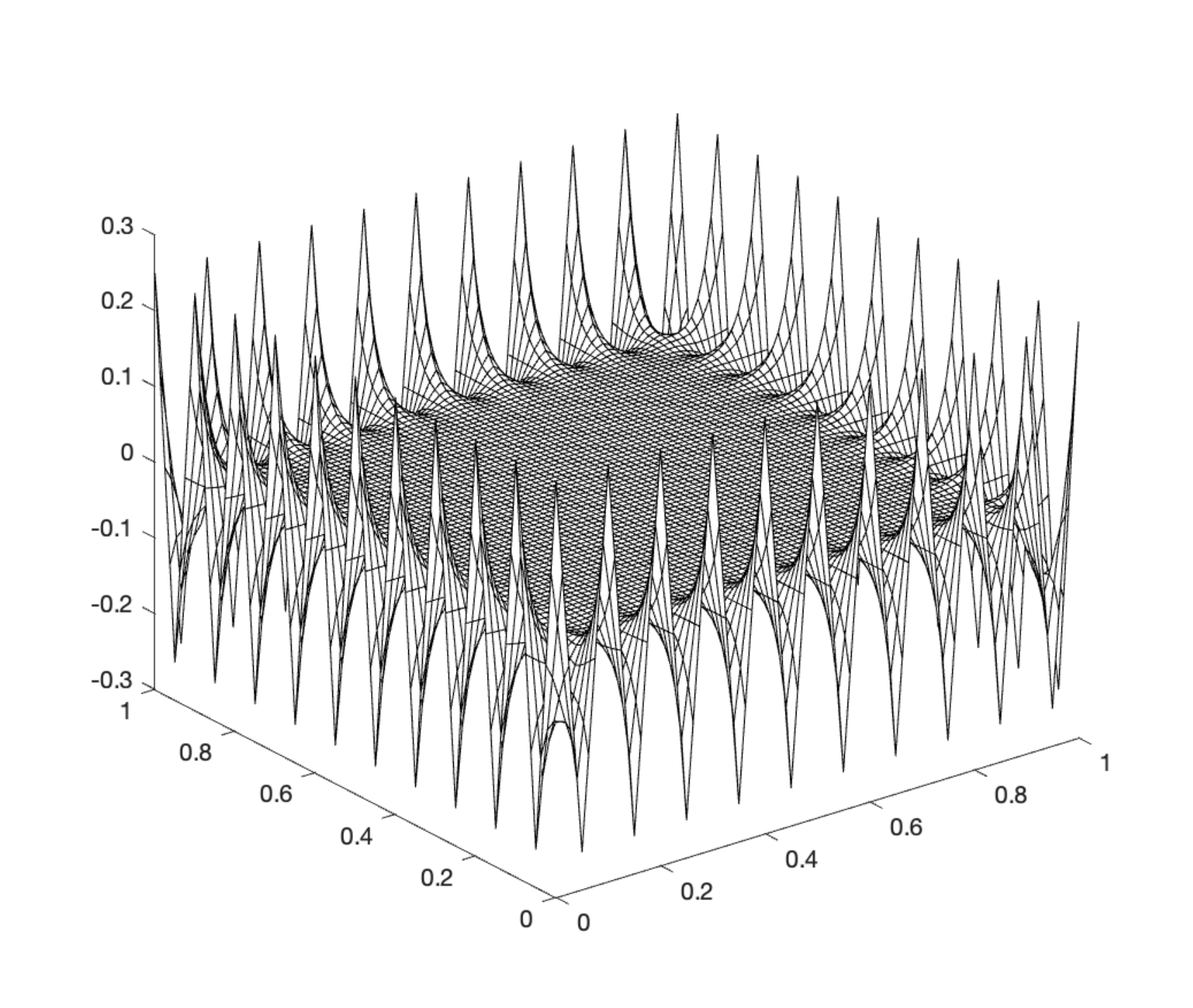}}
    \caption{$\tauVEM=10$}
    \label{fig:80x80-tau=10}
   \end{subfigure}\hfill
   \begin{subfigure}{0.48\textwidth}
     \centering
    {\includegraphics[trim = 1.6cm 1cm 1.6cm 1.8cm, clip, scale=0.4]{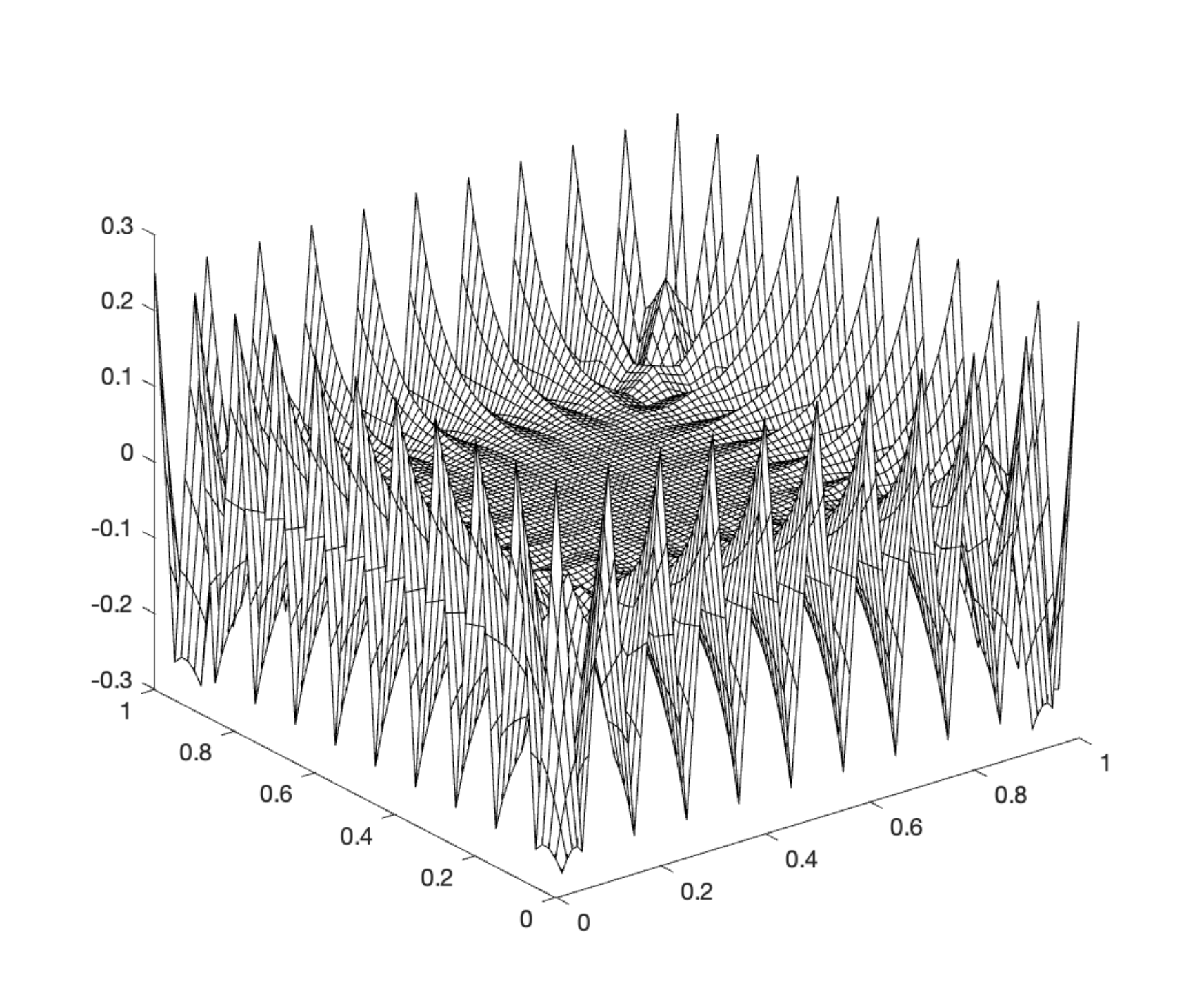}}
    \caption{$\tauVEM=100$}
    \label{fig:80x80-tau=100}
   \end{subfigure}
    \caption{Numerical solution using different $\tauVEM$ for
            the VEM on a $80 \times 80$ mesh.}
            \label{fig:80x80-tau-VEM}
\end{figure}

\subsection{Numerical comparison between isoparametric FEM and VEM}

Consider the following diffusion problem in the unit square, $\Omega=(0,1)^2$:
\begin{equation}\label{eq:problem2}
\left\{
\begin{aligned}
-\text{div}\,(\vK\nabla u) &= f\quad\text{ in }\Omega \\
                 u &= g\quad\text{ on }\partial\Omega
\end{aligned}
\right.
\end{equation}
where $\vK(x,y)=\begin{bmatrix} 1+y^2 & -xy \\ -xy & 1+x^2\end{bmatrix}$ and $f$ and $g$ are chosen
in such a way that
\begin{equation*}
    u(x,y) = x^3 - xy^2 + x^2y -xy + x^2 -x + y -1 + \sin 
    (5x) \sin (7y) + \log(1+x^2+y^4)
\end{equation*}
is the exact solution. In each element, $\vK$ is approximated
with its value at the barycenter of the element.
We compare isoparametric FEM and VEM on two meshes of irregular
quadrilaterals that are shown in Fig.~\ref{fig:meshes}.
The numerical results on the two meshes 
are presented in Figs.~\ref{fig:uh1ISOVEM} 
and~\ref{fig:uh2ISOVEM}. We observe that the solutions
of FEM and VEM are  indistinguishable with
proximal errors in the $L_\infty$ norm.

\begin{figure}[!htb]
   \begin{subfigure}{0.48\textwidth}
     \centering
    {\includegraphics[trim = 0cm 0cm 0cm 0cm, clip, scale=0.4]{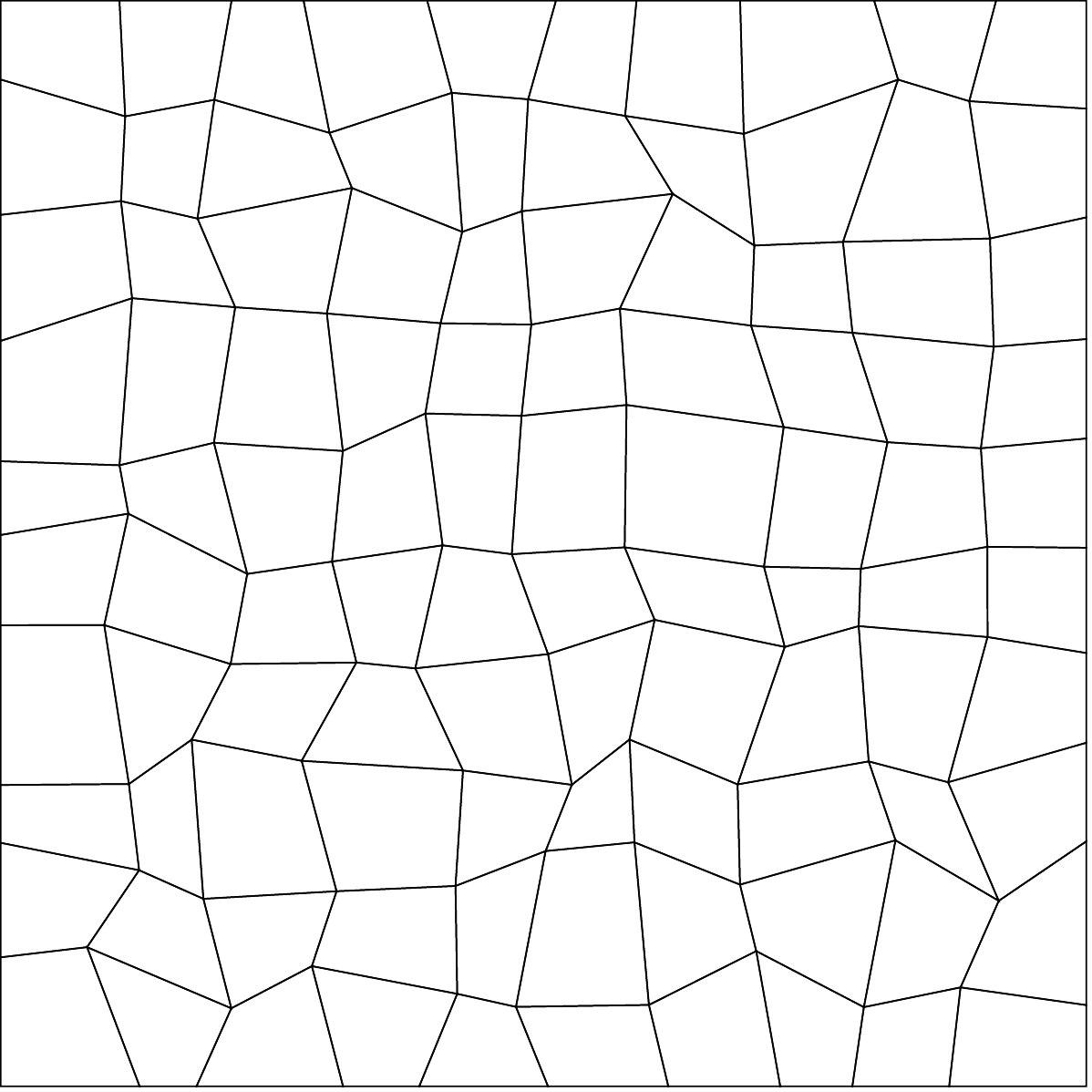}}
    \subcaption{$10\times10$ mesh; $h_\text{mean}=0.16$}
    \label{fig:mesh1}
   \end{subfigure}\hfill
   \begin{subfigure}{0.48\textwidth}
     \centering
    {\includegraphics[trim = 0cm 0cm 0cm 0cm, clip, scale=0.4]{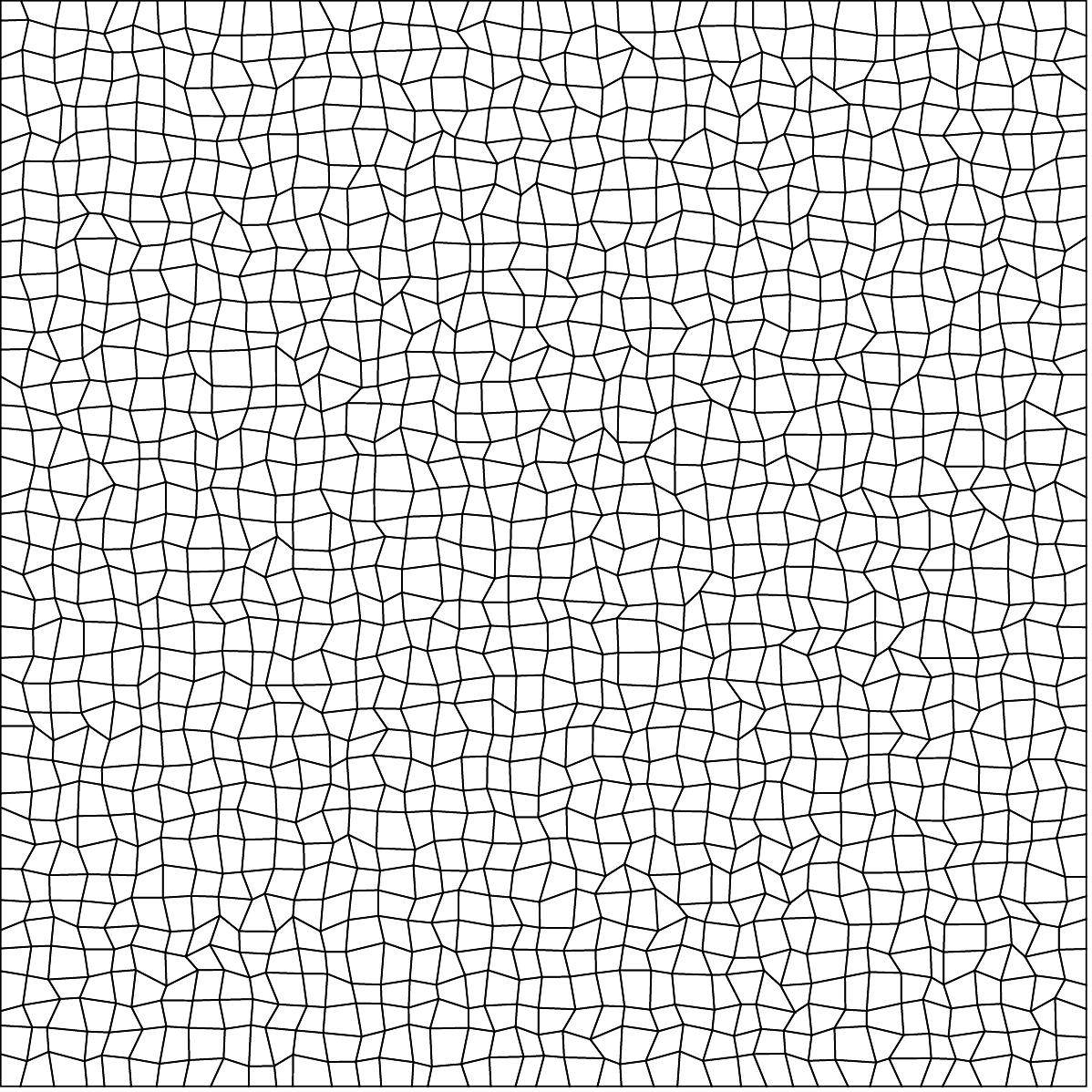}}
    \subcaption{$40\times40$ mesh; $h_\text{mean}=0.04$}
    \label{fig:mesh2}
   \end{subfigure}
   \caption{Finite element meshes used in the FEM and VEM
            to solve~\eqref{eq:problem2}.}
            \label{fig:meshes}
\end{figure}
%
%
\begin{figure}[!htb]
   \begin{subfigure}{0.48\textwidth}
     \centering
    {\includegraphics[trim = 1.6cm 1cm 1.6cm 1.8cm, clip, scale=0.4]{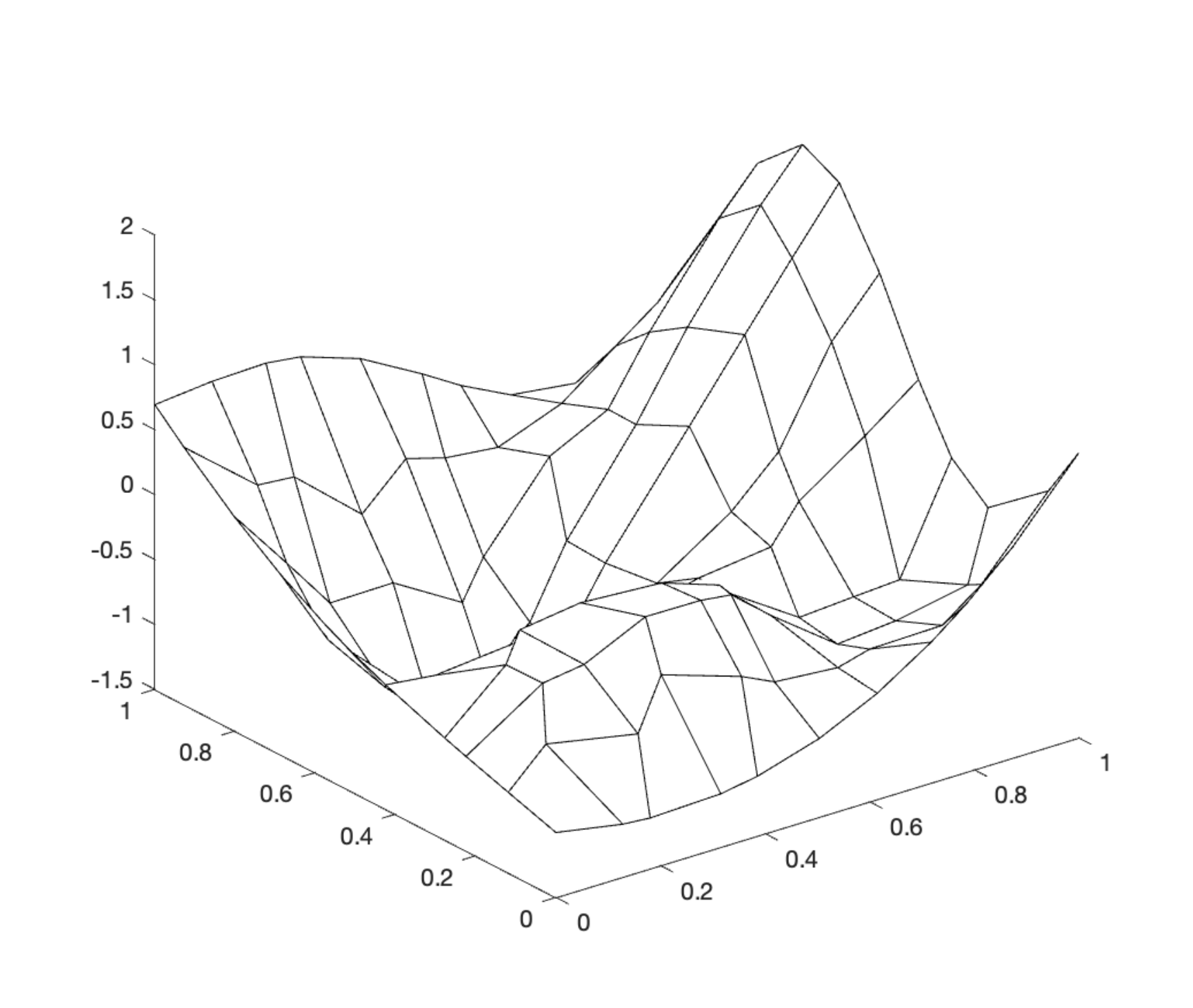}}
    \caption{FEM, $| u - u^h|_\infty = 3.6\times10^{-2}$}
    \label{fig:uh1ISO}
   \end{subfigure}\hfill
   \begin{subfigure}{0.48\textwidth}
     \centering
    {\includegraphics[trim = 1.6cm 1cm 1.6cm 1.8cm, clip, scale=0.4]{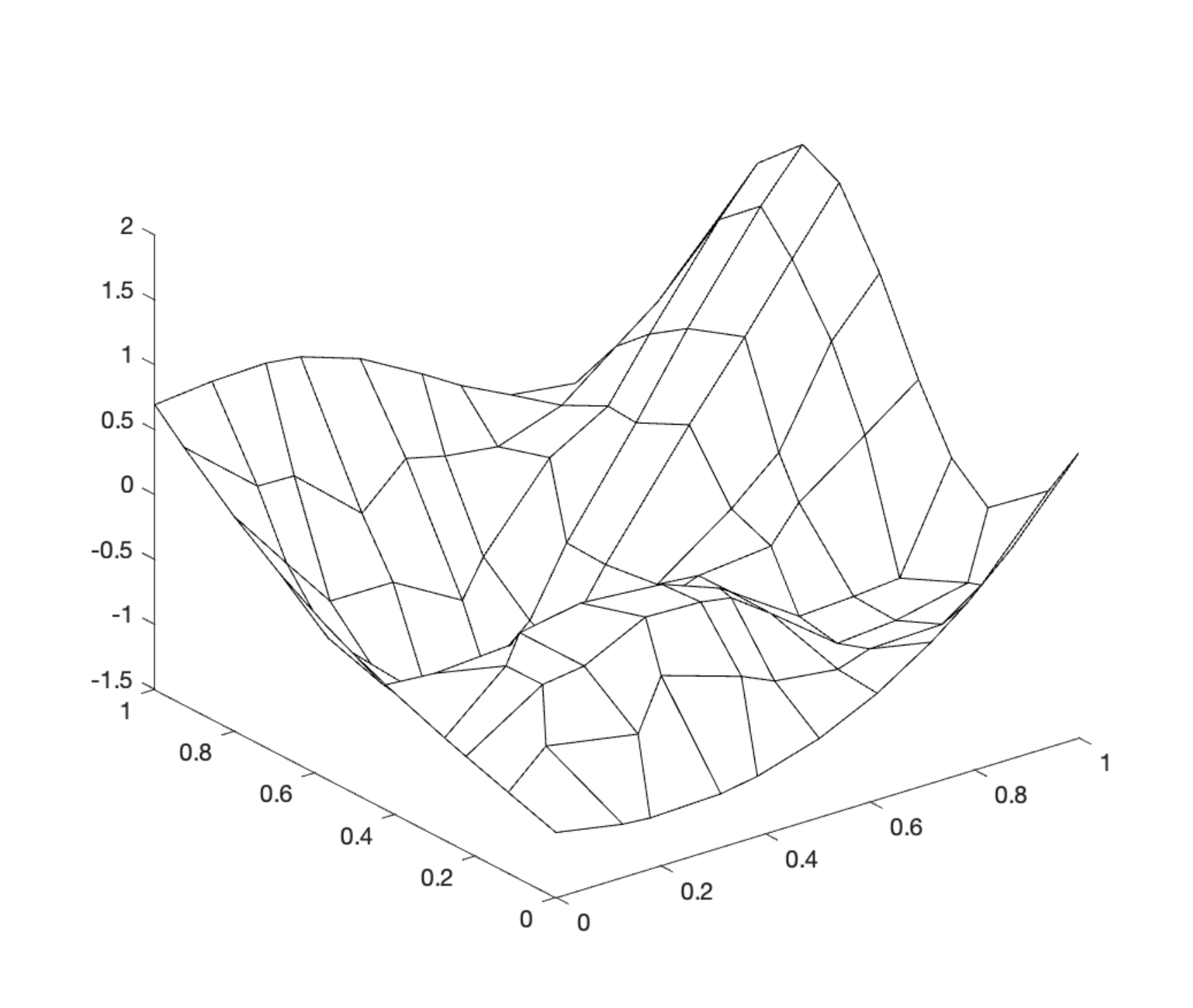}}
    \subcaption{VEM, $| u - u^h|_\infty = 2.4\times10^{-2}$}
    \label{fig:uh1VEM}
   \end{subfigure}
   \caption{Error at vertices for isoparametric FEM and VEM on the mesh shown in Fig.~\protect\ref{fig:mesh1}.}\label{fig:uh1ISOVEM}
\end{figure}
%
%
\begin{figure}[!htb]
   \begin{subfigure}{0.48\textwidth}
     \centering
    {\includegraphics[trim = 1.6cm 1cm 1.6cm 1.8cm, clip, scale=0.4]{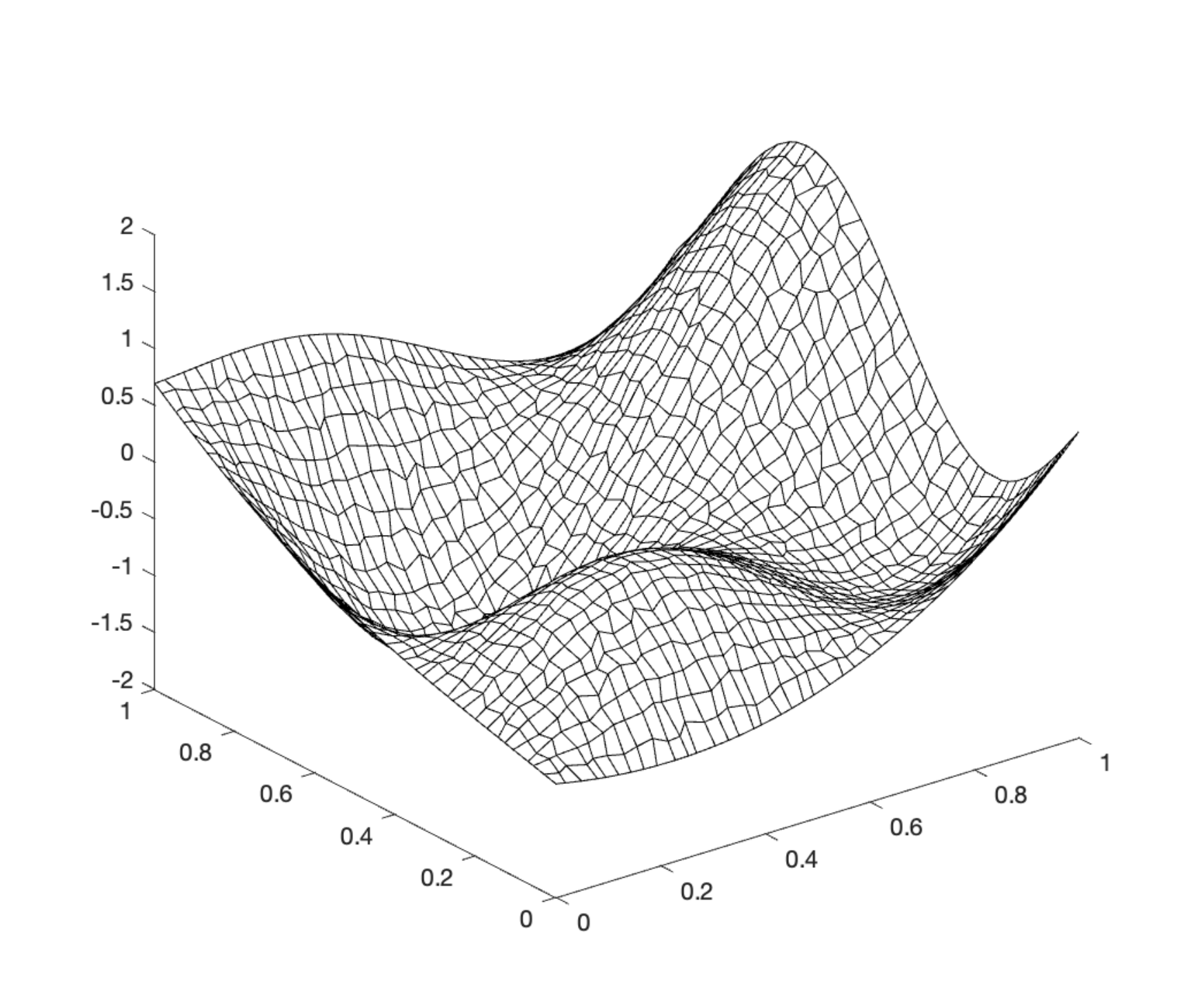}}
    \subcaption{FEM, $|u - u^h|_\infty = 2.6\times10^{-3}$}
    \label{fig:uh2ISO}
   \end{subfigure}\hfill
   \begin{subfigure}{0.48\textwidth}
     \centering
    {\includegraphics[trim = 1.6cm 1cm 1.6cm 1.8cm, clip, scale=0.4]{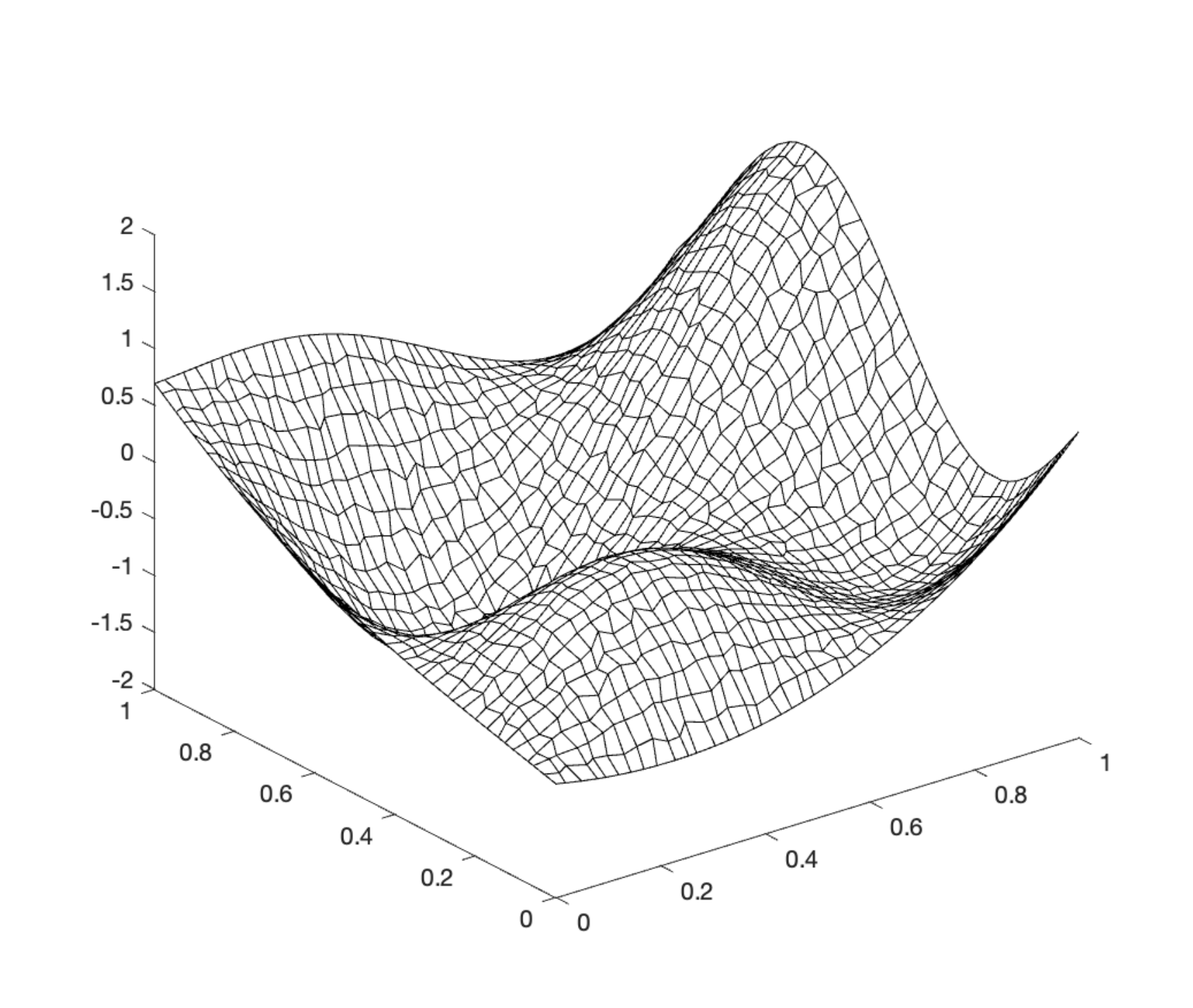}}
    \subcaption{VEM, $|u - u^h|_\infty = 2.5\times10^{-3}$}
    \label{fig:uh2VEM}
   \end{subfigure}
   \caption{Error at vertices for isoparametric FEM and VEM on the mesh shown in Fig.~\protect\ref{fig:mesh2}.}
   \label{fig:uh2ISOVEM}
   
\end{figure}

\section{Conclusions}
Our main conclusion from this study is that VEM is 
relatively insensitive with respect to the stabilization parameter
$\tauVEM$. A caveat in reaching this inference is 
that the mesh must be fine enough so
that the hourglass modes are no longer present
in the exact solution. As we have shown in the numerical experiment, coarse mesh accuracy and stability can be compromised if the hourglass function is a Dirichlet boundary condition that is exactly imposed. 
As part of future work, we plan to extend our analysis to distorted quadrilaterals, and to also include mean value coordinates and harmonic coordinates in our study.

\end{document}